\documentclass[a4paper,reqno]{amsart}
\usepackage[utf8]{inputenc}
\usepackage{enumerate}
\usepackage{amsmath, amssymb,amsthm}
\usepackage{mathtools}
\usepackage{leftidx}
\numberwithin{equation}{section}
\usepackage[left=3cm, right=3cm]{geometry}
\usepackage{hyperref}
\usepackage{xcolor}
\definecolor{my-black}{rgb}{0,0,0}
\definecolor{my-blue}{rgb}{0,0,0.8}
\definecolor{my-red}{rgb}{0.8,0,0} 
\definecolor{my-green}{rgb}{0,0.5,0}
\hypersetup{colorlinks,
    linkcolor	= {my-blue},
    citecolor	= {my-red},
    urlcolor		= {my-black}
}
\usepackage{graphicx}
\usepackage{tikz-cd}

\theoremstyle{plain} 

\newtheorem{lemma}{Lemma}[section]
\newtheorem{theorem}{Theorem}
\newtheorem{corollary}{Corollary}[section]
\newtheorem{proposition}{Proposition}[section]

\newtheorem*{theorem*}{Theorem} 
\newtheorem*{theoremA}{Theorem A}

\theoremstyle{definition} 
\newtheorem{remark}{Remark}[section]

\theoremstyle{remark}
\newtheorem*{remark-non}{Remark}

\DeclareMathOperator{\id}{id}
\DeclareMathOperator{\real}{Re}
\DeclareMathOperator{\imag}{Im}

\DeclareMathOperator{\SL}{SL}

\DeclareMathOperator{\PSL}{PSL}

\DeclareMathOperator{\diam}{diam}

\newcommand{\R}{\mathbb{R}}
\renewcommand{\P}{\mathbb{P}}
\newcommand{\C}{\mathbb{C}}
\newcommand{\N}{\mathbb{N}}

\newcommand{\Z}{\mathbb{Z}}
\newcommand{\Ss}{\mathcal{S}}

\renewcommand{\H}{\mathbb{H}}

\title{Perturbed Fourier uniqueness and interpolation results in higher dimensions}
\author{Jo{\~a}o P.G. Ramos and Martin Stoller}

\address{ETH Z{\"u}rich, Department of Mathematics, R{\"a}mistrasse 101, CH-8092, Z{\"u}rich, Switzerland}
\email{joao.ramos@math.ethz.ch}

\address{B{\^{a}}timent des Math{\'e}matiques, EPFL, Station 8, CH-1015 Lausanne,  Switzerland}
\email{martin.stoller@epfl.ch}

\begin{document}

\begin{abstract}
We obtain new Fourier interpolation and -uniqueness results in all dimensions, extending methods and results by the first author and M. Sousa \cite{Ramos-Sousa-perturbation} and the second author \cite{Stoller-fip-spheres}. 
We show that the only Schwartz function which, together with its Fourier transform, vanishes on surfaces close to the origin-centered spheres whose radii are square roots of integers, is the zero function. In the radial case, these surfaces are spheres with perturbed radii, while in the non-radial case, they can be graphs of continuous functions over the sphere. As an application, we translate our perturbed Fourier uniqueness results to perturbed Heisenberg uniqueness for the hyperbola, using the interrelation between these fields introduced and studied by Bakan, Hedenmalm, Montes-Rodriguez, Radchenko and Viazovska \cite{HMRV}.
\end{abstract}
\maketitle

\section{Introduction}
Let $f : \R \to \C$ be a suitably smooth function. One fundamental question, whose answers have consequences in several different areas of mathematical analysis, is the following: How to retrieve, with the ``minimal'' amount of information on $f$ and on its Fourier transform $\widehat{f},$ the values of $f$ (almost everywhere)? 

One of the first and most classical examples is the Shannon--Whittaker interpolation formula: If $\widehat{f}$ is compactly supported on an interval, which we suppose without loss of generality to be $[-1/2,1/2],$ then the \emph{pointwise} interpolation formula 

\[
f(x) = \sum_{n \in \Z} f(n) \frac{\sin (\pi(x-n))}{\pi(x-n)}
\]
holds. Here and henceforth, we normalize for the Fourier transform of $f \in L^1(\R^d)$ by
\[
\mathcal{F}f(\xi) := \hat{f}(\xi) := \int_{\R^d}{f(x) e^{-2 \pi i \langle x ,\xi \rangle }dx}.
\]
In fact, if $\widehat{f} \in L^2(\R)$ has support on $[-1/2,1/2],$ one has that $\{f(n)\}_{n \in \Z} \in \ell^2(\Z),$ and thus it may be proved that the series on the right-hand side of the Shannon--Whittaker formula converges uniformly on compact sets of $\C$, from which one may deduce that $f$ is an entire function of suitable exponential type. The condition that $\hat{f}$ is compactly supported is quite restrictive, however. Indeed, the celebrated \emph{Paley-Wiener theorem} implies that all functions $f \in L^2(\R)$ satisfying the above formula have compactly supported Fourier transforms, and thus represent only a relatively small amount of all $L^2$ functions. 

Recently, an increasing effort has been put on finding interpolation formulas and proving uniqueness results that use information about values of the function \emph{and} its Fourier transform (with no restriction on the support). The historical starting point of interpolation formulas that involve knowledge of values of the function and its Fourier transform is the interpolation formula of Radchenko and Viazovska \cite{RV}, which states that any sufficiently well-behaved\footnote{Schwartz functions are certainly admissible, but see  \cite[Thm. 7.1]{BRS} for a stronger result. Moreover, see \cite[Thm. 7]{RV} for the case of odd functions.} even function $f : \R \rightarrow \C$ can be completely 
recovered from the values $f(\sqrt{n}), \hat{f}(\sqrt{n})$ for non-negative integers $n$. 
This theorem has inspired and influenced many recent works in the field. We mention two of them.

On the one hand, the first author and M. Sousa \cite{Ramos-Sousa-perturbation} devised functional-analytic methods to perturb the Radchenko--Viazovska formula, as well as other interpolation formulas, such as the classical Whittaker--Shannon formula. They proved a perturbed interpolation result for nodes of the form $\sqrt{n + \varepsilon_n}$ where the perturbations $\varepsilon_n$ must obey a bound of the form $|\varepsilon_n| \leq \delta n^{-5/4}$. Along the way, they also proved new exponential bounds on the Radchenko--Viazovska interpolating functions. 

On the other hand, the second author proved in \cite{Stoller-fip-spheres} an interpolation result that generalizes the one by Radchenko and Viazovska to higher dimensions. The main results of \cite[Thm 1, Thm 3]{Stoller-fip-spheres} show that any Schwartz function on $\R^d$ can be completely recovered from its restrictions, and the restrictions of its Fourier transform, to all origin-centered spheres of radius $\sqrt{n}$, where $n \geq 0$ is an integer. It shares some, but not all features with the one-dimensional Ranchenko--Viazovska formula, as the main number-theoretic input is different, and also many technicalities. We explain some of these technicalities in more details after stating Theorem \ref{thm:non-radial-uniqueness-intro} below. 

Given these results,  it was natural to ponder whether the functional-analytic methods by Ramos--Sousa are applicable in the non-radial setting from \cite{Stoller-fip-spheres}. More precisely, could a perturbed interpolation result in higher dimensions exist, in which the ``nodes" are ``perturbed spheres", that is to say, surfaces close to the spheres of radius $\sqrt{n}$?


In this paper, we answer such questions using new ideas on the functional-analytic side stemming from the methods in \cite{Ramos-Sousa-perturbation} to obtain the results indicated in the abstract. All the proofs of the new results in this paper rely on a natural generalization of the Radchenko--Viazovska formula, which was recently established by Bondarenko--Radchenko--Seip in \cite{BRS}. While this is not not explicitly written down in \cite{BRS}, it follows form more general results in that paper. To be more precise, with some additional technical work,  we will prove the following result in \S \ref{sec:fip-radial}.  In the statement of the theorem and elsewhere, we denote by $\mathcal{S}(\R^d)$ the Schwartz space and by $\mathcal{S}_{\text{rad}}(\R^d)$ the subspace of radial Schwartz functions, equipped with the usual topologies.

\begin{theoremA}\label{thm:thmA}
Let $d \geq 1$ be an integer. Set  $k = d/2$ and define the non-negative integers 
\[
\nu_{-}(k) =  \left \lfloor \frac{k+2}{4}  \right \rfloor, \qquad \nu_{+}(k) = \left \lfloor \frac{k+4}{4} \right \rfloor .
\]
Then, for each $\epsilon \in \{ \pm 1\}$,  there exist even Schwartz functions $b_{k,n}^{\epsilon } \in \Ss(\R)$, indexed by integers $n \geq \nu_{\epsilon}(k)$, with the following properties:
\begin{enumerate}[(i)]
\item Seen as radial Schwartz functions on $\R^d$, we have $ \mathcal{F}(b_{k,n}^{\epsilon}) = \epsilon b_{k,n}^{\epsilon}$ and the functions  $a_{k,n}, \tilde{a}_{k,n} \in \Ss(\R)$, defined by
\begin{equation}\label{eq:definition-an}
a_{k,n}(r) = \frac{b_{k,n}^{+}(r) + b_{k,n}^{-}(r)}{2}, \qquad \tilde{a}_{k,n}(r) = \frac{b_{k,n}^{+}(r) - b_{k,n}^{-}(r)}{2},
\end{equation}
are such that, for every $f \in \Ss_{\text{rad}}(\R^d)$, we have
\begin{equation}\label{eq:interpolation-formula-radial-in-thmA}
f(x) = \sum_{n = \nu_{-}(k)}^{\infty}{\left( a_{k,n}(|x|)f(\sqrt{n}) + \tilde{a}_{k,n}(|x|) \hat{f}(\sqrt{n}) \right)},
\end{equation}
where the partial sums on the right converge in the Schwartz topology towards $f$. 
\item There exist absolute constants $c_1, c_2, c_3 \geq 0$, such that for all decay rates $\beta \geq 2k +2$ and all integers  $n \geq \nu_{\epsilon}(k)$ we have
\begin{equation}\label{eq:final-bound-bkn-in-thmA}
\sup_{r \geq 0}{|(1+r^{\beta})b_{k,n}^{\epsilon}(r)|} \leq   \tilde{g}(\beta) e^{c_1 k + c_2 \beta+c_3}\Gamma(\beta/2-k+1) (1+n)^{\beta/2+k+1},
\end{equation}
where $\tilde{g}(\beta) = \max{(1, (\beta/ 2 \pi e)^{\beta/2})}$.
\item There is a constant $C_k >0$ depending only on $k$ such that 
\begin{equation}\label{eq:final-pointwise-bound-bkn-in-thmA}
|b_{k,n}^{\epsilon}(r)| \le C_k (n+1)^{k+1} e^{-c_1|r|/\sqrt{n+1}}
\end{equation}
for all integers $n \geq \nu_{\epsilon}(k)$ and all $r \in \R$. 
\end{enumerate}
\end{theoremA}
As already mentioned, all of the main ideas for the proof of (parts (i) and (ii) of) Theorem \ref{thm:thmA} come from \cite{RV, BRS}; we only need to keep track of the dependence on the dimension $d = k/2$ in various estimates that arise to prove part (ii). Part (iii) is based on the same idea that is used to prove \cite[Theorem~1.5]{Ramos-Sousa-perturbation} in the case $d = 1$.

We now turn to the newer contributions of this paper which can be seen as ``consequences" of Theorem A combined with several other ideas, which we sketch below.  Our first theorem is a perturbation of the radial interpolation formula \eqref{eq:interpolation-formula-radial-in-thmA}. To state it,  let $V^s(\R^d)$ denote the space of all $f \in L^1(\R^d)$ with the property that $(1+|x|^s) f(x)$ and $(1+|\xi|^s) \hat{f}(\xi)$ are both integrable and denote by $V_{\text{rad}}^s(\R^d)$ the subspace of radial functions. We return to these spaces in \S \ref{sec:perturbation-radial}.

\begin{theorem}\label{thm:radial-perturbation-in-intro}
Fix $d \geq 1$, $s \geq 1$, $\eta >0$  and two sequences of real numbers $\varepsilon_n, \hat{\varepsilon}_n$, indexed by integers $n \geq 0$. Then there is $\delta = \delta(s,d,\eta)>0$ such that, if
\begin{equation}\label{eq:smallness-assumption}
|\varepsilon_n|+|\hat{\varepsilon}_n| \leq \delta (1+n)^{-d-(s/2)-2-\eta},
\end{equation}
for all $n \geq 0$, the following holds true. There are functions $h_{d,n}, \tilde{h}_{d,n} \in V^s_{\text{rad}}(\R^d)$ such that $h_{d,n} = \tilde{h}_{d,n} = 0$ for $n < \nu_{-}(d/2)$ and such that, for all integers $s' \geq (d+1)(s+2d+5 + 2\eta) $ and all $f \in V_{\text{rad}}^{s'}(\R^d)$, we have
\begin{equation}\label{eq:pertrubed-radial-interpolation-formula}
f= f(\varepsilon_0) h_{d,0}+\sum_{n=1}^{\infty}{f(\sqrt{n + \varepsilon_n}) h_{d,n}} + \hat{f}(\hat{\varepsilon}_0) \tilde{h}_{d,0}+ \sum_{n=1}^{\infty}{\hat{f}(\sqrt{n + \hat{\varepsilon}_n}) \tilde{h}_{d,n}},
\end{equation}
where the series converges absolutely in $V_{\text{rad}}^{s}(\R^d)$.
\end{theorem}
Theorem \ref{thm:radial-perturbation-in-intro} is proved by interpreting \eqref{eq:pertrubed-radial-interpolation-formula} as a perturbation of the identity operator  on a suitable space $V_{\text{rad}}^{s}(\R^d)$, whree the identity operator is expressed via the formula \eqref{eq:interpolation-formula-radial-in-thmA} (at least on a dense subspace). The method is similar but in a certain sense simpler than the one devised in \cite{Ramos-Sousa-perturbation}, which relied on the fact that the Radchenko--Viazovska formula is a \emph{free} interpolation formula in the space of radial Schwartz functions on $\R$. That is, the assignment $f \mapsto ((f(\sqrt{n}))_{n \geq 0}, (\hat{f}(\sqrt{n}))_{n \geq 0})$ defines an isomorphism between that space and a space of pairs of rapidly decaying sequences. This property allowed the first author and M. Sousa to work on a suitable Hilbert space of pairs of sequences of complex numbers. In the present paper, we do not make use of such a translation of the problem to a space of sequences, which, on the one hand, greatly simplifies the proof, but, on the other hand, potentially weakens the conclusion (due to potentially less precise estimates).

Let us now turn to the setting of non-radial functions and surfaces ``close to spheres" as indicated in the abstract. Here we prove the following Fourier uniqueness result. 
\begin{theorem}\label{thm:non-radial-uniqueness-intro}
Let $d \geq 2$ and for each $n \geq 1$, let $\varepsilon_n, \hat{\varepsilon}_n : S^{d-1} \rightarrow \R$ be bounded  measurable functions and let  $\varepsilon_0, \hat{\varepsilon}_0 \in \R^d$ be vectors. Define subsets  $S, \widehat{S} \subset \R^d$ by\footnote{In the definition of $S, \widehat{S}$, the (possible) terms with $n = 0$ in the union have to be interpreted as $\{ \varepsilon_0 \}$ and $\{\hat{\varepsilon}_0\}$.}
\[
S= \bigcup_{n \geq \nu_{-}(d/2)}{\{ \sqrt{n}\zeta + \varepsilon_n(\zeta) \zeta \,:\, \zeta \in S^{d-1} \}}, \quad \widehat{S}= \bigcup_{n \geq \nu_{-}(d/2)}{\{ \sqrt{n}\zeta + \hat{\varepsilon}_n(\zeta) \zeta \,:\, \zeta \in S^{d-1} \}}.
\]
Then there exists an absolute constant $c >0$ and a constant $\delta = \delta_d >0$ such that the following holds. If $|\varepsilon_0| + | \hat{\varepsilon_0}| \leq \delta$ and if 
\[
\sup_{\zeta \in S^{d-1}}{|\varepsilon_n(\zeta)|} + \sup_{\zeta \in S^{d-1}}{|\hat{\varepsilon}_n(\zeta)|} \leq \delta n^{-10n - (5/2)d - c} \qquad \text{for all } n \geq 1,
\]
then the only Schwartz function $f \in \mathcal{S}(\R^d)$ such that $f|_S = 0$ and $\widehat{f}|_{\widehat{S}} = 0$, is $f = 0$.
%
%
%
\end{theorem} 
\begin{remark}
As we will explain by the end of the proof of Theorem \ref{thm:non-radial-uniqueness-intro}, a slight variation of the proof allows us to establish existence of \emph{discrete} Fourier uniqueness sets contained in the union $\cup_{n \ge 1} \sqrt{n} S^{d-1},$ with $O(n^{C n})$-many points lying on the sphere of radius $\sqrt{n} S^{d-1}$. See Remark \ref{rmk:discrete-uniqueness-sets} and Corollary \ref{cor:discrete-uniqueness-sets} for details.
\end{remark}
Let us briefly explain what goes into the proof of Theorem \ref{thm:non-radial-uniqueness-intro} and simultanously compare it to the framework of \cite{Stoller-fip-spheres}. The latter article constructs for all $d \geq 5$  smooth kernels $A_n^d, \tilde{A}_n^{d}: \R^d \times S^{d-1} \rightarrow \C$ such that for all $f \in \mathcal{S}(\R^d)$ and all $x \in \R^d$, one has 
\begin{equation}\label{eq:non-radial-interpolation-formula-intro}
f(x) = \sum_{n=1}^{\infty}{\int_{S^{d-1}}{A_n^d(x, \zeta)}f( \sqrt{n}\zeta) d\zeta}  + \sum_{n=1}^{\infty}{\int_{S^{d-1}}{\tilde{A}_n^d(x, \zeta)}\widehat{f}( \sqrt{n}\zeta) d\zeta}
\end{equation} 
with pointwise absolute convergence and uniform convergence on compact sets avoiding the origin. Given the functional analytic argument with which we prove Theorem \ref{thm:radial-perturbation-in-intro}, a natural approach to perturb the above formula would be to replace $f(\sqrt{n}\zeta)$ by $f(\sqrt{n}\zeta + \varepsilon_n(\zeta)\zeta)$ (and similarly for $\hat{f}$) on the right hand side of \eqref{eq:non-radial-interpolation-formula-intro} and to prove that the resulting expression defines a bounded operator on some space $V^s(\R^d)$ which is close to the identity. However, this does not seem to work because, although the functions $x \mapsto A_n(x, \zeta)$ are smooth and somewhat controlled, they are not seen to decay enough or even belong to  $L^2(\R^d)$. Indeed, the  radial functions denoted ``$b_{p,n}(r)$" in \cite{Stoller-fip-spheres} out of which the kernels $A_n^d$ are built as infinite sums, have only very little decay in $r$, as made precise in \cite[Proposition~8.1]{Stoller-fip-spheres}. Furthermore, an approach based on a suitable space of pairs of sequences of functions on $S^{d-1}$, in the spirit of \cite{Ramos-Sousa-perturbation}, does not seem to work either, as the image of the natural map from $\mathcal{S}(\R^d)$ to a suitable space of sequences of functions on $S^{d-1}$, has infinite dimensional cokerenel (see  \cite[Proposition~7.1]{Stoller-fip-spheres}) and no description of the image is known (in contrast to the related works \cite{RV} or \cite{CKMRV-universal}).

Thus, a different approach is needed, which can roughly be summarized as follows. Instead of perturbing \eqref{eq:non-radial-interpolation-formula-intro} we will (in a certain sense) perturb \emph{another} interpolation formula for non-radial Schwartz functions, which \emph{follows} from the interpolation formula \eqref{eq:interpolation-formula-radial-in-thmA} in Theorem \ref{thm:thmA} and a harmonic analysis result from \cite[Corollary~2.1]{Stoller-fip-spheres}. We write down this formula in equation \eqref{eq:double-series-formula} in \S \ref{sec:perturbation-sphere}. It involves two double series of integrals over the sphere, which may not converge absolutely.  Still, by formally interchanging the sums and the integrals and making use of a key feature the basis functions $a_{d/2,n}, \tilde{a}_{d/2,n}$ in Theorem \ref{thm:thmA} (namely that the first $\asymp d$ of them all vanish; see the definition of $\nu_{-}(d/2)$), we are able to meaningfully write down an operator, which can be thought of as a perturbation of an identity operator, expressed via a (hypothetical) formula that looks just like \eqref{eq:non-radial-interpolation-formula-intro}. If the perturbations are small enough, we can prove that this operator is invertible and its injectivity allows us to derive the uniqueness result of Theorem \ref{thm:non-radial-uniqueness-intro}. 

 
We shall finish the paper with a brief discussion on an application of Theorem \ref{thm:radial-perturbation-in-intro}. In \cite[Section~7,~Open~Problem~(a)]{HM-R},  H. Hedenmalm and A. Montes-Rodr\'iguez pose the following question. Consider the hyperbola $\Gamma = \{ (x_1,x_2) \in \R^2 \colon x_1 x_2 = 1\}$ and a \emph{perturbed} lattice cross
\begin{equation}\label{eq:perturbed-lattice-cross}
\Lambda =  \{(\alpha n + \varepsilon_n,0)\, : \, n \in \Z\} \cup \{(0,\beta n + \widehat{\varepsilon_n})\, : \, n \in \Z\}.
\end{equation}
Is $(\Gamma,\Lambda)$ a \emph{Heisenberg uniqueness pair} whenever $0 < \alpha \beta \le 1$?  We recall that $(\Gamma,\Lambda)$ is a Heisenberg uniqueness pair if for all finite complex Borel measures $\mu$ in the plane $\R^2$ that are supported on the curve $\Gamma$ and absolutely continuous with respect to arc-length measure, one has $\widehat{\mu}|_{\Lambda} = 0 \Rightarrow \mu = 0$.  Here, we normalize the Fourier transform of $\mu$, as in \cite{HM-R}, by $\widehat{\mu}(\xi) = \int_{\R^2}{e^{ \pi i  \langle x, \xi \rangle} d\mu(x)}$, $\xi \in \R^2$.

We partially answer this question in the case $\alpha = \beta =1$\footnote{implying a result in the more general case $\alpha\beta =1$ via dilations.} in the spirit of recent, related work by the first author and F. Gon\c calves on the case of the parabola \cite{Goncalves-Ramos}, where the authors define the notion of \emph{weak} Heisenberg uniqueness pairs. They call a pair $(\Gamma,\Lambda)$ a weak Heisenberg uniqueness pair if the vanishing condition in the definition of a Heisenberg uniqueness pair holds in a suitable class of sufficiently regular measures. 

In a similar spirit,  we consider the measures $\mu = \mu_{f}$ attached to \emph{odd} functions $f \in V^s(\R)$, and characterized by 
\begin{equation}\label{eq:definition-of-mu_f}
\int_{\R^2}{ \varphi \,d\mu_f } = \int_{\Gamma}{ \varphi \,d\mu_f } =\int_{\R^{\times}} \varphi(t,1/t) t^3 f(t) \sqrt{1+t^{-4}}\, dt \quad \text{ for all }  \quad \varphi \in C_c^{\infty}(\R^2).
\end{equation}
Here, the factor $\sqrt{1+t^{-4}}$ reflects the geometry of $\mu_f$ related to the arc-length measure on $\Gamma$ (while the factor $t^3$ and the condition that $f$ is odd are of technical nature). Note that since $f$ is odd, the functions $\xi_1 \mapsto \widehat{\mu_f}(\xi_1, 0)$ and $\xi_2 \mapsto \widehat{\mu_f}(0, \xi_2)$ are both even, so it is natural to consider only the part $\Lambda^{+} = \Lambda \cap [0,+\infty)^2$ of $\Lambda$ and correspondingly only with sequences of perturbations $\varepsilon_n, \hat{\varepsilon}_n$ indexed by $n \in \Z_{\geq 0}$ and such that $n + \varepsilon_n \geq 0$, $n + \hat{\varepsilon}_n \geq 0$ for all $n$.

\begin{theorem}\label{thm:application-hup} 
There exist constants $\delta >0$ and $s_0 \geq 1$ such that the following holds true. For all sequences of real numbers $\varepsilon_n, \hat{\varepsilon}_n$, $n \geq 0$, satisfying $|\varepsilon_n| + |\hat{\varepsilon}_n| \leq \delta n^{-7}$ for all $n \geq 1$  and $\varepsilon_0 = \hat{\varepsilon}_0 = 0$, the perturbed lattice cross $\Lambda$, given by \eqref{eq:perturbed-lattice-cross}, with parameters $\alpha = \beta =1$ and the hyperbola $\Gamma$ form a Heisenberg uniqueness pair $(\Gamma, \Lambda)$ for the set of measures $\{ \mu_f\,:\,f \in V^{s_0}(\R) \text{ odd}\}$. That is, one has $\widehat{\mu_f}|_{\Lambda} = 0 \Rightarrow \mu_f = 0$, for all odd $f \in V^{s_0}(\R)$.
\end{theorem}

The main idea in the proof is to construct an auxiliary radial function $\Phi$ on $\R^4$, originally from  \cite{HMRV}, which vanishes, together with its Fourier transform, on spheres of radii $\sqrt{n + \varepsilon_n}$ and $\sqrt{n + \hat{\varepsilon}_n}$ respectively and thus falls under the scope of Theorem \ref{thm:radial-perturbation-in-intro}. A bit more specifically, if we put $g(t) = t^3 f(t) \sqrt{1+t^{-4}}$, then $\Phi$ is given as the composition of the squared Euclidean norm $\R^4 \rightarrow [0, \infty)$ with a one-dimensional Fourier transform of an \emph{anti-derivative} of $g$.  The somewhat artificial conditions on $f$ are imposed so that we can apply Theorem \ref{thm:radial-perturbation-in-intro} to $\Phi$ in the case $d = 4$. Thus we do not expect these  restrictions to be essential for a more general version of our result to hold, although new ideas seem to be necessary for that purpose.

\subsection{Acknowledgements} The authors are grateful to Mateus Sousa for several discussions during the early stages of this manuscript and to Danylo Radchenko for remarks and suggestions on its final version. We would also like to thank the anonymous referee for valuable comments and suggestions. J.P.G.R. acknowledges financial support through the ERC grant agreement No. 721675 “Regularity and Stability in Partial Differential Equations (RSPDE)".

\section{Proof of Theorem A}\label{sec:fip-radial}
The purpose of this preparatory section is to prove Theorem A, upon which the proofs of Theorems \ref{thm:radial-perturbation-in-intro} and \ref{thm:non-radial-uniqueness-intro} are based (Theorem \ref{thm:application-hup} is based on Theorem \ref{thm:radial-perturbation-in-intro}). The reader who is willing to take Theorem A for granted and is mainly interested in seeing how it is applied in the proofs our main theorems  can also directly go to \S \ref{sec:perturbation-radial} (for Theorem \ref{thm:radial-perturbation-in-intro}), \S \ref{sec:perturbation-sphere} (for Theorem \ref{thm:non-radial-uniqueness-intro}) or to  \S \ref{sec:application} (for Theorem \ref{thm:application-hup}),  depending upon interest.
%
\subsection{Set up and proof of part (i) in Theorem A  }\label{sec:prelim-for-radial}
For the remainder of \S \ref{sec:fip-radial}, we let $k$ denote a half integer $\geq 1/2$ and we let $\epsilon \in \{ \pm 1 \}$ denote a sign. The interpretation of $k$ is that $2k$ is the dimension of $\R^{2k}$ on which we will consider our radial functions, as in Theorem A (although, for most of the analysis, $k$ could be any nonnegative real number as in \cite{BRS}). We recall the definition of the numbers $\nu_{\epsilon}(k)$ given in Theorem A and now also define the auxiliary numbers $\mu_{\epsilon}(k)  \in [-7/8,0]$ by 
\begin{equation}\label{eq:definition-mu-nu}
\nu_{-}(k) = \left \lfloor \frac{k+2}{4}\right \rfloor, \quad  \mu_{-}(k)  = -\left \lbrace \frac{k+2}{4} \right \rbrace, \quad
\nu_{+}(k) = \left \lfloor \frac{k+4}{4}\right \rfloor, \quad \mu_{+}(k) = -\left \lbrace \frac{k+4}{4} \right \rbrace.
\end{equation}
Here, $\lbrace x \rbrace, \lfloor x \rfloor$ denote the the fractional- and integer part of $x \in \R$ respectively so that $x= \lbrace x \rbrace + \lfloor x \rfloor$. 

On the upper half-plane $\H = \{ \tau \in \C \,:\,  \imag(\tau) > 0 \}$, we determine a holomorphic  logarithm $\tau \mapsto \log(\tau/i)$ by requiring that its value at $\tau = i$ is zero and we define (complex) powers $(\tau/i)^{-k}$ accordingly. For $(r, z) \in \R \times \H$, let $\varphi_r(z):= e^{\pi i z r^2}$. Then $z \mapsto (r \mapsto \varphi_r(z))$ defines a continuous map $\H \rightarrow \mathcal{S}_{\text{rad}}(\R^1)$ that is of moderate growth, in the sense that its post-composition with any continuous semi-norm is of moderate growth (see \cite[\S 4.1]{CKMRV-universal} or \cite[\S 3]{BRS} for the definition of moderate growth that we use). Therefore, by \cite[Theorem 3.1]{BRS}, there are two-periodic analytic functions $\tau \mapsto F_k^{\epsilon}(\tau,r)$ (one for each $r,k$ and $\epsilon$) of moderate growth, satisfying
\begin{equation}\label{eq:functional-equation-F}
F_k^{\epsilon}(\tau,r)- \epsilon(\tau/i)^{-k}F_k(-1/\tau, r) = \varphi_r(\tau) - \epsilon (\tau/i)^{-k}\varphi_r(-1/\tau)
\end{equation}
and admitting the Fourier expansion $F_k^{\epsilon}(\tau,r) = \sum_{n=\nu_{\epsilon}(k)}^{\infty}{b_{k,n}^{-\epsilon}(r)e^{\pi i n \tau}}$ in which the coefficients are given by
\begin{equation}\label{eq:def-bk-eps}
b_{k,n}^{-\epsilon}(r) = \frac{1}{2} \int_{iy-1}^{iy+1}{F_k^{\epsilon}(\tau,r) e^{- \pi i n \tau}d\tau},
\end{equation}
for any $n \in \Z$, independently of $y >0$. The equation \eqref{eq:def-bk-eps} is the definition of $b_{k,n}^{\epsilon}$ that appears in Theorem A. Since the collection of radial Schwartz functions $\R^{2k} \ni x \mapsto \varphi_{|x|}(z) + \epsilon (z/i)^{-k}\varphi_{|x|}(-1/z)$ for $z \in \H$ generates a dense subspace of the $\epsilon$-eigenspace of the Fourier transform acting on $\mathcal{S}_{\text{rad}}(\R^{2k})$ (this  follows from \cite[Lemma 2.2]{CKMRV-universal}) we deduce from the functional equations \eqref{eq:functional-equation-F} that for all $f$ in that space and all $x \in \R^{2k}$ we have 

\begin{equation}\label{eq:radial-interpolation-formula-in-eigenspace}
f(x) = \sum_{n = \nu_{\epsilon}(k)}^{\infty}{f(\sqrt{n})b_{k,n}^{\epsilon}(|x|)}.
\end{equation}
Indeed, the definition of functions $b_{k,n}^{\epsilon}$ is precisely such that the functional equation \eqref{eq:functional-equation-F} says that \eqref{eq:radial-interpolation-formula-in-eigenspace} holds for all $f$ of the form $f(x) = \varphi_{|x|}(z) + \epsilon (z/i)^{-k}\varphi_{|x|}(-1/z)$ for all $z \in \H$. On the other hand, since for each fixed $x \in \R^d$ with $|x| = r$, the sequence $n \mapsto b_{k,n}^{\epsilon}(r)$ grows polynomially with $n$ (as we will prove and as follows in a qualitative form already from the moderate growth assumption; see below),  both the right and left-hand side of \eqref{eq:radial-interpolation-formula-in-eigenspace} define tempered distributions.

Thus, by writing a general $f \in \mathcal{S}_{\text{rad}}(\R^{2k})$ as $f = (f+\hat{f})/2 + (f- \hat{f})/2$ and applying the formula \eqref{eq:radial-interpolation-formula-in-eigenspace} to each summand, we get the interpolation formula \eqref{eq:interpolation-formula-radial-in-thmA}, stated in part (i) of Theorem A, with $a_{k,n}$ and $\tilde{a}_{k,n}$ defined as in Theorem A. 

It may not be immediately clear why \cite[Theorem 3.1]{BRS} implies that each $r \mapsto b_{k,n}^{\epsilon}(r)$ is a Schwartz function, but this is the case and will also be implicitly proven in our estimates below (by combining Proposition \ref{prop:decay} with \eqref{eq:final-bound-bkn-in-thmA}). Moreover, each $b_{k,n}^{\epsilon}$ is an $\epsilon$-eigenvector for the Fourier transform on $\R^{2k}$. In any case, the cited Theorem \emph{directly} implies $b_{k,n}^{\epsilon}(r)$ grows at most polynomially in $n$, for each fixed $r$ and this alone suffices to establish \eqref{eq:radial-interpolation-formula-in-eigenspace} (and \eqref{eq:interpolation-formula-radial-in-thmA}) with point-wise absolute convergence.

\subsubsection{Modular kernels} Here, we recall from \cite{BRS} how the functions $F_k^{\epsilon}$ appearing in \ref{eq:functional-equation-F} are constructed as integral transforms of certain modular kernels and how these kernels are defined. We will need them in the next subsection. 

Let $\Gamma_{\theta} \leq \PSL_2(\Z)$ denote the theta group. It is the image in $\PSL_2(\Z)$ of the subgroup of $\SL_2(\Z)$ generated by $S,T^2 \in \SL_2(\Z)$, where
\[
S:= \begin{pmatrix}
0 &-1\\
1 & 0
\end{pmatrix}, \qquad T := \begin{pmatrix}
 1 & 1\\
 0 & 1
\end{pmatrix}
\]
denote the well-known generators for $\SL_2(\Z)$. We use the open fundamental domain $\mathcal{D} = \{ z \in \H \,:\, |z| >1 ,-1 < \real(z) < 1 \}$ for $\Gamma_{\theta} \backslash \H$. We need the three theta functions
\[
\Theta_2(z) = \sum_{n \in \Z}{e^{\pi i (n+1/2)^2 z}}, \qquad \Theta_3(z) = \theta(z) = \sum_{n \in \Z}{e^{\pi i n^2 z}}, \qquad \Theta_4(z) = \sum_{n \in \Z}{(-1)^ne^{\pi i n^2 z}},
\]
the modular lambda invariant $\lambda = \Theta_2^4/ \Theta_3^4$ and the functions
\[
J(z) := J_{+}(z) :=  \frac{16}{\lambda(z)(1- \lambda(z))} = 16 \frac{ \Theta_3(z)^8}{\Theta_2(z)^4 \Theta_4(z)^4}, \qquad J_{-}(z) := 1-2 \lambda(z) .
\]
We recall that none of $\Theta_2^4, \Theta_3^4, \Theta_4^4$, $\lambda$ has a zero on the upper half-plane, that they take real, positive values on $i\R_{>0}$ and real values on the boundary of $\mathcal{D}$. For $\kappa \in \R$ we define $\theta^{\kappa}$ via $\log{\theta}(\tau) := \int_{i\infty}^{\tau}{\theta'(z)/\theta(z) dz}$ and based on it, we define the automorphy factors\footnote{This is the reciprocal of the one used in \cite{BRS}.}  
\[
j_{\theta,k}(\gamma, z) := \theta^{2k}(\gamma z)/\theta^{2k}(z) \quad \text{for} \quad (\gamma, z) \in \PSL_2(\R) \times \H.
\] 
Let $\chi_{\epsilon} : \Gamma_{\theta} \rightarrow \{\pm 1\}$ denote the group homomorphism satisfying $\chi_{\epsilon}(T^2) = 1$ and $\chi_{\epsilon}(S) = \epsilon$. For any function $f$ defined on $\H$ with values in a complex vector space and any $\gamma \in \Gamma_{\theta}$, define the function $f|_{k}^{\epsilon}\gamma $ by $f|_{k}^{\epsilon}\gamma (z) = \chi_{\epsilon}(\gamma) j_{\theta,k}(\gamma,z)^{-1}f(\gamma z)$.  

The formula $\lambda' = \pi i \lambda(1- \lambda) \theta^4$ implies $J' = - \pi i J_{-}J_{+} \theta^4$. We use it to write the kernels on \cite[Page~18]{BRS} as
\begin{align}
\mathcal{K}_k^{+}(\tau,z) &= \frac{1}{\pi i } \frac{J'(z)}{J(z)-J(\tau)} \frac{\theta^{2k}(\tau)}{\theta^{2k}(z)} \frac{J(z)^{\nu_{+}(k)-1}}{J(\tau)^{\nu_{+}(k)-1}} ,\label{eq:kernels-with-derivative-plus} \\
\mathcal{K}_k^{-}(\tau,z) &= \frac{1}{\pi i } \frac{J'(z)}{J(z)-J(\tau)} \frac{\theta^{2k}(\tau)}{\theta^{2k}(z)} \frac{J(z)^{\nu_{-}(k)-1}}{J(\tau)^{\nu_{-}(k)-1}} \frac{J_{-}(\tau)}{J_{-}(z)}. \label{eq:kernels-with-derivative-minus}
\end{align}
For each $\tau \in \mathcal{D}$ and $r \in \R$ we have, by \cite[Proposition~3.3]{BRS},
\begin{equation}\label{eq:standard-integral-representation-F}
F_k^{\epsilon}(\tau,r) = \frac{1}{2}\int_{-1}^{1}{\mathcal{K}_k^{\epsilon}(\tau,z) \varphi_r(z) dz},
\end{equation}
where the path from $-1$ to $1$ is taken along a semicircle, oriented clockwise. For each fixed $\tau \in \mathcal{D}$, the function $r \mapsto F_k^{\epsilon}(\tau,r)$ is a Schwartz function, since $z \mapsto \mathcal{K}_k^{\epsilon}(\tau,r)\varphi_r(z)$ is a continuous, Schwartz-space valued map which extends continuously by zero at the cusps $-1, 1$. That is, the limit is  the zero function when these points are approached within $\overline{\mathcal{D}}$.

Before we turn to estimating the Fourier coefficients $b_{k,n}^{\epsilon}(r)$ in the next section, let us show here that these are real-valued. By \eqref{eq:def-bk-eps}, it suffices to show that $\overline{F_k^{\epsilon}(\tau,r)} = F_k^{\epsilon}(- \overline{\tau},r)$ for all $\tau \in \H, r \in \R$. For this, it suffices to  show that $F_k^{\epsilon}(\tau,r)$ is real-valued on the imaginary axis, by the Schwarz reflection principle. In fact, it suffices to show $F_k^{\epsilon}(it_0, r) \in \R$ for all $t_0 >1$ (say).  To that end, we fix $\tau \in \mathcal{D}$ and first apply a contour shift to \eqref{eq:standard-integral-representation-F} to obtain the following expression\footnote{To see this, apply the residue theorem to the boundary of  $\{z \in \mathcal{D}\,:\, \imag(z) < y, |z-1| > \varepsilon, |z+1|> \varepsilon\}$, for some fixed $y > \max{(\imag(\tau),1)}$ and then let $\varepsilon \rightarrow 0$.}
\[
F_k^{\epsilon}(\tau,r) = e^{\pi i \tau r^2} + \int_{0}^{y}{\mathcal{K}_k^{\epsilon}(\tau,1+it)\sin(\pi r^2)e^{- \pi tr^2}dt} + \frac{1}{2}\int_{iy-1}^{iy+1}{\mathcal{K}_k^{\epsilon}(\tau,z)e^{\pi i z r^2}dz},
\]
where $y > \max{(\imag(\tau),1)}$. Now we take $\tau = it_0$ with $t_0 > 0$, conjugate the above identity and conclude by using that for all  $x  \in \R, t_0 > 0$ we have \[
\overline{\mathcal{K}_k^{\epsilon}(it_0, x+iy)} = \mathcal{K}_k^{\epsilon}(it_0, -x+iy), \quad \mathcal{K}_k^{\epsilon}(it_0,1+it) \in \R.
\]
\begin{remark-non}
By \cite[Proposition 3.2]{BRS} one also has $b_{k,n}^{\epsilon}(r) = \int_{-1}^{1}{g_{k,n}^{\epsilon}(z) \varphi_r(z) dz}$ for certain weakly holomorphic modular  forms $g_{k,n}^{\epsilon}$ of weight $k$ and character $\chi_{\epsilon}$ for $\Gamma_{\theta}$. This representation can be used to give an alternative, slightly more direct proof of the fact that $b_{k,n}^{\epsilon}(r) \in \R$ via a contour shift similar to the above. More significantly, it  implies that, for fixed $n \geq \nu_{\epsilon}(k)$, all of the values $b_{k,n}^{\epsilon}(\sqrt{m})$, $m \geq \nu_{\epsilon}(k)$, are zero except for one, since the numbers $b_{k,n}^{\epsilon}(\sqrt{m})$ are coefficients of the principal part of the Laurent expansion at infinity of $g_{k,n}^{\epsilon}$, which is seen to have a correspondingly simple form. The present paper does \emph{not} rely on such facts, but they should imply (as in \cite{RV}) that \eqref{eq:interpolation-formula-radial-in-thmA} is a free interpolation formula and thus, in principle, allow an application of the same perturbation techniques as in \cite{Ramos-Sousa-perturbation} to perturb the formula \eqref{eq:interpolation-formula-radial-in-thmA}.
\end{remark-non}
\subsection{Proof of part (ii) in Theorem A}\label{sec:estimates-for-basis-functions}
This subsection is devoted to the  proof of part (ii) in Theorem A, which stated an estimate for $\sup_{r \geq 0}{|(1 + r^{\beta})b_{k,n}^{\epsilon}(r)|}$, that makes the dependence on all parameters explicit. This will be a bit long and complicated, so we give a brief overview before we start. We will divide the proof of the bounds on the functions $b_{k,n}^{\epsilon}$ as in \cite{BRS}. In \S \ref{sec:estimates-in-fundamental-domain} we first bound $(1+ r^{\beta})^{}F_k^{\epsilon}(\tau,\varphi_r)$ for $\tau \in \mathcal{D}$, the fundamental domain. Then  in \S \ref{sec:cocycle-estimates} we use the functional equations repeatedly to deduce bounds in all of $\H$. We conclude in \S \ref{sec:conclusion} by applying the triangle inequality at a suitable height $y$ to the integral \eqref{eq:def-bk-eps} defining $b_{k,n}^{\epsilon}(r)$. Moreover, throughout \S \ref{sec:estimates-for-basis-functions}, we work with the following notations:
\begin{itemize}
\item a half-integer $k \geq 1/2$ and a sign $\epsilon \in \{\pm 1 \}$,
\item a positive real number $\beta >0$, thought of as a decay rate,
\item a real number $r \geq 0$, thought of as a radius. 
\end{itemize}
Since it will suffice for our later applications and is technically convenient, we assume throughout that $\beta \geq 2k+2$. We omit some of these parameters in our notation and abbreviate
\begin{itemize}
\item the slash action $f|_k^{\epsilon} \gamma$ to $f|\gamma$,
\item kernels $\mathcal{K}_k^{\epsilon}(\tau,z)$ to $\mathcal{K}(\tau,z)$,
\item the numbers $\mu_{\epsilon}(k), \nu_{\epsilon}(k)$ (defined in \eqref{eq:definition-mu-nu}) to $\mu_{\epsilon}, \nu_{\epsilon}$,
\item and $(1+r^{\beta})F_k^{\epsilon}(\tau,r)$ to $F(\tau)$.
\end{itemize}
For the remainder of \S \ref{sec:estimates-for-basis-functions}, we adopt the convention that  a \emph{constant} is a positive real number that does \emph{not} depend on $k, \beta,$ or $r$, but may depend on (the sometimes hidden) sign $\epsilon$. We shall work with a $1$-cocycle $\gamma \mapsto \psi_{\gamma}$, $\Gamma_{\theta} \rightarrow \{ f: \H \rightarrow \C \}$. That is, a collection of functions indexed by the elements of $\Gamma_{\theta}$ satisfying $\psi_{AB} = \psi_B + \psi_A|B$ for all  $A,B \in \Gamma_{\theta}.$ In our case, the cocyle is determined on generators by
\[
\psi_{T^2}(\tau) = 0, \qquad \psi_{S}(\tau) = (1+r^{\beta}) \left( \varphi_r(\tau)- \epsilon (\tau/i)^{-k} \varphi_r(-1/\tau) \right),
\]
where we recall that $\varphi_r(z) = e^{\pi i z r^2}$. Then $F-F|\gamma = \psi_{\gamma}$ for all $\gamma \in \Gamma_{\theta}$ (see \cite[\S 6.2]{BRS} for a justification). One-variable calculus shows that for all $\tau \in \H$,
\begin{equation}\label{eq:universal-estimate-gaussian}
(1+r^{\beta}) |e^{\pi i \tau r^2}| \leq  1+g(\beta) \imag(\tau)^{-\beta/2}, \quad \text{where} \quad g(\beta):=\left( \frac{\beta}{2 \pi e} \right)^{\beta/2}.
\end{equation}
It follows that
\begin{align}
|\psi_S(\tau)| &\leq (1+|\tau|^{-k}) + g(\beta) \imag(\tau)^{-\beta/2}(1+|\tau|^{\beta-k}) \notag \\
			   &\leq \tilde{g}(\beta) \left( 1 + |\tau|^{-k} + \imag(\tau)^{-\beta/2}(1+|\tau|^{\beta-k}) \right), \quad \text{where} \quad  \tilde{g}(\beta):= \max{(1, g(\beta))}. \label{eq:universal-estimate-psi_S}
\end{align}

\subsubsection{Estimates in the fundamental domain}\label{sec:estimates-in-fundamental-domain}
We start with estimates in the fundamental domain, for which we first record a couple of asymptotic relations of the building blocks $J, \theta$ of the kernels $\mathcal{K}$.
\begin{lemma}\label{lem:asymptotic-relatoins-near-1}
Fix a compact set $\Omega \subset  \overline{\mathcal{D}} \cup S \overline{ \mathcal{D}}$ containing $-1$, bounded away from the point $0$. Then, writing $w = J(z)$ and confining $z \in \Omega$, we have
\begin{align}
 \imag(z)^{-1} \asymp_{\Omega} \log{(e+1/|w|)}, \label{eq:imag-in-terms-of-w}\\
 |\theta(z)|^2 \asymp_{\Omega} |w|^{1/4}\log{(e+1/|w|)}. \label{eq:theta-squraed-in-terms-of-w}
\end{align}
If $\Omega$ is sufficiently small, then we can replace $\log{(e+1/|w|)}$ by $\log{(1/|w|)}$ in \eqref{eq:imag-in-terms-of-w} and \eqref{eq:theta-squraed-in-terms-of-w}.
\end{lemma}
\begin{proof}
By continuity, it suffices to establish \eqref{eq:imag-in-terms-of-w} and \eqref{eq:theta-squraed-in-terms-of-w} in the case where $\Omega$ is sufficiently close to $-1$. Then, if $z \in \overline{\mathcal{D}} \cup S \overline{ \mathcal{D}}$ has $\imag(z) \leq 1/2$ and $\real(z) \leq 1/2$, the point $\tilde{z}:= STz = \frac{-1}{z+1}$ satisfies
\begin{equation}\label{eq:asymptotic-relations-uniformizer-at-1}
\frac{1}{2 \imag(z)} \leq \imag(\tilde{z}) = \frac{\imag(z)}{|z+1|^2} \leq \frac{1}{\imag(z)}, \qquad 
\frac{1}{\sqrt{2} \imag(z)} \leq \frac{1}{|z+1|} = |\tilde{z}| \leq \frac{1}{\imag(z)}.
\end{equation}
Additionally, we have $\tilde{z} \rightarrow i \infty$ as $z \rightarrow -1$, in the sense that $\text{Im}(z) \to \infty$ and $|\text{Im}(z)/\text{Re}(z)| \to \infty$ as $z \to -1.$ 

Indeed, if $z = -1 + \omega,$ then $\tilde{z} = \frac{-1}{\omega},$ and thus $\text{Re}(\tilde{z}) = - \frac{\text{Re}(\omega)}{|\omega|^2},$ while $\text{Im}(\tilde{z}) = \frac{\text{Im}(\omega)}{|\omega|^2}.$  This and \eqref{eq:asymptotic-relations-uniformizer-at-1} show that $|\tilde{z}|, \text{Im}(\tilde{z}) \to \infty$ as $z \to -1,$ while, as $\omega = z + 1 \in \overline{\mathcal{D}} \cup S \overline{\mathcal{D}},$ we have $\left|\frac{\text{Im}(\omega)}{\text{Re}(\omega)}\right| \to \infty,$ which finishes the claim. 

From the transformation rules of $\Theta_2,\Theta_3,\Theta_4$ and $J$ we get
\begin{align*}
J(z) &= J(TS \tilde{z}) = 16\frac{\Theta_3(TS \tilde{z})^8}{\Theta_2(TS \tilde{z})^4 \Theta_4(TS \tilde{z})^4} =-16 \frac{\Theta_2(\tilde{z})^8}{\Theta_4(\tilde{z})^4 \Theta_3(\tilde{z})^{4}}, \\
\frac{\theta(z)^8}{J(z)} &= \frac{\Theta_3(TS \tilde{z})^8}{J(TS \tilde{z})}= -( \tilde{z}/i)^{4} \frac{1}{16}\Theta_4(\tilde{z})^4 \Theta_3(\tilde{z})^{4}.
\end{align*}
Using that $\Theta_2(\tilde{z}) \sim 2 e^{\pi i \tilde{z}/4}$, $\Theta_3(\tilde{z}), \Theta_4(\tilde{z}) \sim 1$ as $ \tilde{z} \rightarrow i \infty$, we get $J(z) \sim 2^{12}e^{2 \pi i \tilde{z}}$, as $z \rightarrow -1$, from which the desired relations follow together with \eqref{eq:asymptotic-relations-uniformizer-at-1}.
\end{proof}
Besides Lemma \ref{lem:asymptotic-relatoins-near-1} we also need the following Lemma, which is similar to \cite[Lemma 6.1]{BRS}, but a bit more explicit.
\begin{lemma}\label{lem:pinelis}
For real numbers $h \in (0,1/e]$, $\delta \in (0, h]$, $\mu \in (-1, 0]$ and $b \geq 0$, define
\begin{equation}\label{eq:def-H-mu-b}
H_{\mu, b}(\delta) := \int_{0}^{h}{\frac{x^{\mu}}{(\delta^2 + x^2)^{1/2}}\log(1/x)^{b}dx}.
\end{equation}
If $b \geq 1$ and $\log(1/\delta)(\mu+1) \geq 1$, then
\begin{equation}\label{eq:upper-bound-H-mu-b}
H_{\mu, b}(\delta) \leq 2^{b+2} \Gamma(b+1) \log(1/\delta)^{b+1}(1/\delta)^{|\mu|}.
\end{equation}
\end{lemma}
\begin{proof}
We split the integral as $ \int_{0}^{\delta} + \int_{\delta}^h$. On the first part we use $x^2 + \delta^2 \geq \delta^2$ and on the second we use $x^2 + \delta^2 \geq x^2$. On both parts, we change to the variable $t = \log(1/x)$ and write $H_{\mu,b}(\delta) \leq A + B$, where
\[
 A= \delta^{-1}(\mu+1)^{-b-1} \Gamma(b+1, \log(1/\delta)(\mu+1)), \qquad B = \int_{\log(1/h)}^{\log(1/\delta)}{t^{b}e^{|\mu|t}dt}
\]
and where $\Gamma(a,x) =\int_{x}^{\infty}{e^{-t}t^{a-1}dt}$ denotes the incomplete Gamma function. A result of Pinelis \cite[Theorem~1.1]{Pinelis} asserts that for all $a \geq 2$ and all $x > 0$, we have
\begin{equation}\label{eq:isoif-upper-bound}
\Gamma(a,x) \leq \frac{(x+c_{a})^{a}-x^{a}}{a c_a}e^{-x}, \quad \text{where} \quad c_a := \Gamma(a+1)^{1/(a-1)}.
\end{equation}
Applying this with $a = b+1$ and $x = \log(1/\delta)(\mu+1)$, we get
\[
A \leq (1/\delta)(1/\delta)^{-(\mu+1)} \log(1/\delta)^{b+1} \frac{\left(1 + \tfrac{c_{b+1}}{\log(1/\delta)(\mu+1)}\right)^{b+1}-1}{(b+1)c_{b+1}} \leq (1/\delta)^{|\mu|} \log(1/\delta)^{b+1} 2^{b+1}\Gamma(b+1),
\]
where we used the assumption $(\mu+1)\log(1/\delta) \geq 1$ and crude upper bounds to get the last inequality (we also used $1 \leq c_{b+1}$ for $b \geq 1$). To bound $B$, we use $\log(1/h) \geq 1$ and $e^{t|\mu|} \leq \log(1/\delta)^{|\mu|}$ for $t$ in the integration range and obtain
\[
B \leq (1/\delta)^{|\mu|} \frac{\log(1/\delta)^{b+1}-1}{(b+1)}.
\]
Hence the upper bound for $A$ is larger than that for $B$, and $H_{\mu,b}(\delta) \leq A + B$ implies \eqref{eq:upper-bound-H-mu-b}.
\end{proof}
Now that we have Lemmas \ref{lem:asymptotic-relatoins-near-1}, \ref{lem:pinelis} we can turn to the estimate of $F(\tau)$ in the fundamental domain.
\begin{proposition}\label{prop:bound-F-in-fd}
With notations and conventions as at the beginning of \S \ref{sec:estimates-for-basis-functions}, there exist constants $c_1, c_2, c_3 \geq 0$ such that for all $\tau \in \mathcal{D}$,
\begin{equation}\label{eq:upper-bound-F-in-fd}
|F(\tau)| = |F_k^{\epsilon}(\tau,r)|(1+r)^{\beta} \leq  \tilde{g}(\beta) e^{c_1 k + c_2 \beta + c_3}\Gamma(\beta/2-k+1)(1+\imag(\tau)^{-\beta/2-1}).
\end{equation}
Here, we recall that $\tilde{g}(\beta) = \max{(1, (\beta/2 \pi e)^{\beta/2})}$ and that we assume $\beta \geq 2k + 2$.
\end{proposition}
\begin{proof}
We closely follow the proof of Proposition 6.1 in \cite{BRS} and the closely related proof of Lemma 4 in \cite{RV}  with a few adaptations, keeping closer track of the dependence on $k$ in the estimates. We start with two preliminary simplifications concerning the set of $\tau \in \mathcal{D}$ for which \eqref{eq:upper-bound-F-in-fd} has to be established.
\begin{enumerate}[(i)]
\item Since both sides of \eqref{eq:upper-bound-F-in-fd} are invariant under the reflection $\tau \mapsto - \overline{\tau}$ and continuous in $\tau$, we may assume that $\real(\tau) \in \mathcal{D}_{\text{left}}$, where
\[
\mathcal{D}_{\text{left}}:= \{ z \in \mathcal{D}\,:\, \real( z) < 0\}, \qquad  \mathcal{D}_{\text{right}}:= \{ z \in \mathcal{D}\,:\, \real( z) >0\}.
\]
\item It suffices to prove \eqref{eq:upper-bound-F-in-fd} for $\tau \in \mathcal{D}$ such that $|\tau-i| \geq 1/4$. Indeed, assuming \eqref{eq:upper-bound-F-in-fd} holds for such $\tau$, it follows for the remaining $\tau \in \mathcal{D}$  (possibly with slightly enlarged constants $c_i$) by applying the maximum modulus principle to to the disc $|\tau-i| \leq 1/4$ combined with the functional equation in the form $F(\tau) = \psi_S(\tau) +\epsilon F(-1/\tau)(\tau/i)^{-k}$. 

\end{enumerate}
Thus, assume henceforth that $\tau \in \mathcal{D}_{\text{left}}$ and that $|\tau-i| \geq 1/4$. We split the integral in \eqref{eq:standard-integral-representation-F} as $\int_{-1}^{i} + \int_{i}^1$ and change variables $z \leftrightarrow -1/z$ on the second piece, giving $F(\tau) =  \frac{1}{2}\int_{-1}^{i}{\mathcal{K}(\tau,z)\psi_S(z)dz}$. Next, we recall that $J|_{\mathcal{D}}$ is injective, that $J(\mathcal{D}_{\text{left}}) = \H$, $J(\mathcal{D}_{\text{right}}) = -\H$ and that $J$, restricted to the quarter circle from $-1$ to $i$, gives a smooth monotone bijection onto $[0,64]$, with $J(i) = 64$. Thus, changing variables $w= J(z)$, $dw = J'(z)dz$ and defining $t(w) :=J^{-1}(w)=z$, we obtain
\begin{align*}
F^{+}(\tau) &= \frac{\theta^{2k}(\tau)}{J(\tau)^{\nu_{+}-1}} \frac{1}{2\pi i }\int_{0}^{64}{ \frac{1}{J(\tau)-w}\frac{w^{\nu_{+}-1}}{\theta^{2k}(t(w))} \psi_S(t(w)) dw},\\
F^{-}(\tau) &= \frac{\theta^{2k}(\tau)}{J(\tau)^{\nu_{-}-1} J_{-}(\tau)} \frac{1}{2\pi i }\int_{0}^{64}{ \frac{1}{J(\tau)-w}\frac{w^{\nu_{-}-1}}{\theta^{2k}(t(w))J_{-}(t(w))} \psi_S(t(w)) dw},
\end{align*}
where we have re-included the dependence on the sign $\epsilon$ in the notation. The difficulty is to control the term $1/(J(\tau)-w)$, which goes to infinity as $\tau$ approaches the left-quarter circle joining $-1$ to $i$. We therefore change the $w$-contour from $[0,64]$ to rectangular path $\ell = \ell_1 \cup \ell_2 \cup \ell_3$, where, for some $h \in (0,1/e]$, to be determined,
\[
\ell_1 = i[-h,0], \qquad \ell_2 = -ih + [0,64], \qquad \ell_3 = 64 +i[-h,0].
\]
On these line segments the following estimates hold
\begin{align}
&|J(\tau)-w|^2 =  |J(\tau)|^2 +|w|^2 -2\real(J(\tau) \overline{w}) \geq |J(\tau)|^2 + |w|^2  &\text{for }w \in \ell_1, \label{eq:lower-bound-on-ell-1}\\
&|J(\tau)-w|^2 = (\real(J(\tau))-\real(w))^2 + (\imag(J(\tau))+h)^2 \geq h^2 & \text{for } w \in \ell_2, \label{eq:lower-bound-on-ell-2}\\
&|J(\tau)-w|  \geq c_0  &\text{for } w \in \ell_3, \label{eq:lower-bound-on-ell-3}
\end{align}
where $c_0 >0$ is an absolute constant, whose existence follows from our assumption that $|\tau-i| \geq 1/4$, making $J(\tau)$ bounded away from $64$. Let $R = R_{\ell}$ denote the rectangle bounded by the $\ell_i$ and $[0,64]$. Note that $\{t(w)\,:\, w  \in R\}$ is a compact subset in the closure of $\mathcal{D}_{\text{right}}$ and that
\[
S \mathcal{D}_{\text{right}} = \{z \in \H\,:\, |z|<1,|z+1/2|<1/2\}.
\]
Using this observation together with the general estimate \eqref{eq:universal-estimate-psi_S} we bound
\[
|\psi_S(t(w))| \leq C_1^{k}C_2^{\beta}\tilde{g}(\beta)(1 + \imag(t(w))^{-\beta/2}), \quad w \in R_{\ell},
\]
with some constants $C_1,C_2 \geq 1$. We deduce that
\begin{align*}
|F^{+}(\tau)| &\leq  \tilde{g}(\beta) C_1^k C_2^{\beta}\left| \frac{ \theta(\tau)^{2k}}{J(\tau)^{\nu_{+}-1}}\right|  \sum_{j=1}^{3}{\int_{\ell_j}{\frac{1}{|J(\tau)-w|}\frac{|w^{\nu_{+}-1}|}{|\theta^{2k}(t(w))|}(1+ \imag(t(w))^{-\beta/2})|dw|     } }, \\
|F^{-}(\tau)| &\leq  \tilde{g}(\beta) C_1^k C_2^{\beta}\left| \frac{ J_{-}(\tau) \theta(\tau)^{2k}}{J(\tau)^{\nu_{-}-1}}\right|  \sum_{j=1}^{3}{ \int_{\ell_j}{\frac{1}{|J(\tau)-w|}\frac{|w^{\nu_{-}-1}|}{|\theta^{2k}(t(w))|}\frac{|w|^{1/2}}{|w-64|^{1/2}} (1+\imag(t(w))^{-\beta/2}) |dw|     } } ,
\end{align*}
where we used $(J_{-})^2 = 1-64/J$ to write $\tfrac{1}{|J_{-}(t(w))|} = \tfrac{|w|^{1/2}}{|w-64|^{1/2}}$ in $F_{-}(\tau)$.
Employing the asymptotic relations \eqref{eq:imag-in-terms-of-w} $\imag(t(w))^{-1} \asymp_{\ell} \log(e + |w|^{-1})$ and \eqref{eq:theta-squraed-in-terms-of-w} $|\theta(t(w))|^2 \asymp_{\ell} |w|^{1/4} \log(e +1/|w|)$ from Lemma \ref{lem:asymptotic-relatoins-near-1}, we can estimate the terms
\begin{align*}
\frac{|w^{\nu_{+}-1}|}{|\theta^{2k}(t(w))|} (1+\imag(t(w))^{-\beta/2}) &\leq C_3^{k}C_4^{\beta}|w|^{\nu_{+}-1-k/4}\log(e + 1/|w|)^{\beta/2-k},\\
\frac{|w^{\nu_{-}-1}|}{|\theta^{2k}(t(w))|}\frac{|w|^{1/2}}{|w-64|^{1/2}} (1+ \imag(t(w))^{-\beta/2}) &\leq C_3^{k}C_4^{\beta} \frac{|w|^{\nu_{-}-1/2-k/4}}{|w-64|^{1/2}}\log(e + 1/|w|)^{\beta/2-k},
\end{align*}
for all $w \in \ell$. We now choose $h$ sufficiently small so that, for $w \in \ell_1$, we can 
 replace $\log(e + 1/|w|)$ by $\log(1/|w|)$ in these  estimates. Inserting them into the estimates for $F^{+}(\tau), F^{-}(\tau)$ from before and using \eqref{eq:lower-bound-on-ell-1}, \eqref{eq:lower-bound-on-ell-2}, \eqref{eq:lower-bound-on-ell-3} as well as integrability of $|w-64|^{-1/2}$ on $\ell_3$, we obtain
\begin{align}
|F^{+}(\tau)| &\leq \tilde{g}(\beta)C_5^k C_6^{\beta} \left| \frac{ \theta(\tau)^{8}}{J(\tau)} \right|^{k/4} |J(\tau)|^{-\mu_{+}}  \left( H_{\mu_{+},\beta/2-k}(|J(\tau)|) +  C_7^k C_8^{\beta} \right), \label{eq:final-bound-F-plus}\\
|F^{-}(\tau)| &\leq \tilde{g}(\beta) C_5^k C_6^{\beta} \frac{|J_{-}(\tau)|}{|J(\tau)|^{1/2}}\left| \frac{ \theta(\tau)^{8}}{J(\tau)} \right|^{k/4} |J(\tau)|^{-\mu_{-}}  \left( H_{\mu_{-},\beta/2-k}(|J(\tau)|) + C_7^k C_8^{\beta} \right),  \label{eq:final-bound-F-minus}
\end{align}
where $H_{\mu,\beta/2-k}$ is the elementary integral defined in  \eqref{eq:def-H-mu-b}. So far we did not make further assumptions on where the point $\tau \in \mathcal{D}_{\text{left}}$ is (besides the standing assumption $|\tau-i| \geq 1/4$). We now consider separately the cases where the point $\tau$ is close to the cusp $-1$  and bounded away from it. 

Precisely, we fix $y_0 \in (0,1]$ such that, if $\imag(\tau) \leq y_0$, then
\begin{equation}\label{eq:making-pinelis-applicable}
\log(1/|J(\tau)|)(\mu_{\epsilon} + 1) \geq \log(1/|J(\tau)|)(-7/8+1) \geq 1.
\end{equation}
For such $\tau$ we also have $|\theta^8(\tau)/J(\tau)| \lesssim \imag(\tau)^{-4}$ and $|J_{-}(\tau)|/|J(\tau)|^{1/2} \lesssim 1$ with implied constants depending at most on $y_0$. Because of \eqref{eq:making-pinelis-applicable} we can now apply Lemma \ref{lem:pinelis} with $b= \beta/2-k \geq 1$, $\mu = \mu_{\epsilon}$, $\delta = |J(\tau)|$, giving
\begin{align*}
|F^{\epsilon}(\tau)| &\leq \tilde{g}(\beta) C_9^{k}C_{10}^{\beta} \imag(\tau)^{-k}|J(\tau)|^{- \mu_{\epsilon}} \left(2^{b+2}\Gamma(b+1)\log(1/|J(\tau)|)^{b+1} |J(\tau)|^{\mu_{\epsilon}} +  C_7^k C_8^{\beta} \right) .\\
&\leq \tilde{g}(\beta) C_9^{k}C_{10}^{\beta} \left(2^{b+2}\Gamma(\beta/2-k+1)C_{11}^{b} \imag(\tau)^{-\beta/2-1} + \imag(\tau)^{-k}|J(\tau)|^{- \mu_{\epsilon}}C_7^{k}C_8^{\beta} \right),
\end{align*}
where we used $\log(1/|J(\tau)|) \lesssim \imag(\tau)^{-1}$. Using $|J(\tau)|^{-\mu_{\epsilon}} \leq 1$ and $\imag(\tau)^{-k} \leq \imag(\tau)^{-\beta/2-1}$ for $\imag(\tau) \leq y_0$, we can bring this into desired form \eqref{eq:upper-bound-F-in-fd}.

Now we consider points $\tau \in \mathcal{D}_{\text{left}}$ satisfying $\imag(\tau) \geq y_0$. Then $|\theta(\tau)^8/J(\tau)|, |J_{-}(\tau)|/|J(\tau)|^{1/2} \lesssim 1$. We estimate $H_{\mu_{\epsilon},b}(|J(\tau)|)$ similarly as the quantity $A$ in the proof of Lemma \ref{lem:pinelis}, namely by
\[
H_{\mu_{\epsilon},b}(|J(\tau)|) \leq \frac{1}{|J(\tau)|}(\mu_{\epsilon} + 1)^{-b-1} \int_{\log(1/h)(\mu_{\epsilon}+1)}^{\infty}{e^{-t}t^{b}dt} \leq \frac{8^{b+1}}{|J(\tau)|} \Gamma(b+1).
\]
Since $1/|J(\tau)| \lesssim 1$, this also gives an estimate of the shape \eqref{eq:upper-bound-F-in-fd} in the region $\imag(\tau) \geq y_0$. By taking the maximum of the estimates in the regions $\imag(\tau) \leq y_0$ and $\imag(\tau) \geq y_0$, we finish the proof.
\end{proof} 

\subsubsection{Cocycle estimates}\label{sec:cocycle-estimates}
Now that we have an estimate of $F(\tau)$ for $\tau \in \mathcal{D}$, the fundamental domain, we wish to derive from it bounds in the entire upper half-plane, by repeatedly applying the functional equations. For this, we closely follow the approach in \cite{RV}.

We start with a few preliminaries in the spirit of geometric group theory. Given real numbers $a, b$ with $a< b$ we write
\[
D(a,b) := \{z \in \H\,:\, |z-(a+b)/2| < (b-a)/2 \},
\]
for the open half-disc with midpoint $(a+b)/2$, bounded by the hyperbolic geodesic joining $a$ and $b$. Given a third point $p \in (a,b)$, we define
\[
\Delta(a,p, b) := D(a, b) \setminus \overline{D(a,p) \cup D(p,b)},
\]
which is a hyperbolic triangle with vertices $a, p, b$. We will use the same notation, i.e. $\Delta(a,p,b),$ to denote the hyperbolic triangle with vertices $a, p, b$ in the additional case of $a,b < p < +\infty$. 

If  $p = \infty \in \P^1(\R) = \partial \H$, then we define
\[
\Delta(a, \infty,b) := \{z \in \H \,:\, \real(z) \in (a, b) \} \setminus \overline{D(a,b)}.
\]
Note that $ \Delta(-1, \infty,1) = \mathcal{D}$ is the fundamental domain for $\Gamma_{\theta} \backslash \H$ we have been using.  If $\Delta  = \Delta(a,p,b)$ has real vertices satisfying either $a < p< b < -1$ or $1 < a< p < b$, i.e. if $\Delta$ is disjoint form the vertical strip $\real(z) \in (-1,1)$), then the triangle $S \Delta = \Delta(-1/a, -1/p, -1/b) \subset D(-1,-1)$ has diameter bounded by
\begin{equation}\label{eq:diameter-comparison}
\diam(S \Delta) = (-1/b)-(-1/a) = \frac{b-a}{ab} = \frac{b-a}{1+(ba-1)} = \frac{\diam(\Delta)}{1+(ba-1)} \leq   \frac{\diam(\Delta)}{1+\diam(\Delta)},
\end{equation}
since 
\[
ba-1 = b-a + ba-1+a-b = \diam(\Delta) + (b+1)(a-1) \geq \diam(\Delta).
\]
We will be interested in the set  $\mathcal{M} \subset \Gamma_{\theta}$ consisting of all elements $M \in \Gamma_{\theta}$ of the form
\begin{equation}\label{eq:general-element-in-M}
M = ST^{2m_n}ST^{m_{n-1}} \cdots ST^{2m_0}, \quad  \quad m_1, \dots, m_{n} \in \Z \setminus \{0 \}, \qquad m_0 \in \Z, \qquad n \geq 0.
\end{equation}
Thus, $M = ST^{2m_0}$ if $n =0$, which then could equal $S$. Let us compute the image of the fundamental domain $\Delta(-1, \infty, 1)$ under $M$ and analyze the diameter of the resulting triangle. We start with
\[
ST^{m_0}\Delta(-1, \infty,1) = S(2m_0-1, \infty, 2m_0 + 1) = \Delta(\tfrac{-1}{2m_0-1}, 0, \tfrac{-1}{2m_0 + 1} ) =: \Delta_0
\]
(which is $\Delta(-1, 0, 1)$ in the case $m_0 = 0$). The next element $T^{2m_1}$ maps $\Delta_0$ outside the vertical strip $\real(z) \in (-1,1)$ without changing its diameter and the inversion $S$ then maps $T^{2m_1}\Delta_0$ back into $D(-1,1)$ giving a triangle $\Delta_1 := ST^{2m_1} \Delta_0$ satisfying $\diam(\Delta_1) \leq \frac{\diam(\Delta_0)}{1+\diam(\Delta_0)}$ by \eqref{eq:diameter-comparison}. We define $\Delta_j := ST^{2m_j}\Delta_{j-1}$, so that $\Delta_n =M \Delta(-1, \infty,1)$. By induction, we see that $\diam(\Delta_j) \leq \frac{2}{2j-1}$ for $1 \leq j \leq n$. This amounts to the computation with \eqref{eq:diameter-comparison} using that $t \mapsto \tfrac{t}{1+t}$ is non-decreasing on $(0,+\infty)$.

Now let us consider a point $\tau \in D(-1,1)$. Let $M \in \Gamma_{\theta}$ and $z \in \overline{\mathcal{D}}$ be such that $Mz = \tau$. Then we must have $M \in \mathcal{M}$ for otherwise, $M = T^{2m'}M'$ with $m' \neq 0$ and $M' \in \mathcal{M}$ and so $\tau =Mz = 2m' + M'z$ but $M'z \in \overline{D(-1,1)}$ by what we have seen above, a contradiction since $m' \neq 0$. 

So let us write $M = ST^{2m_n} \cdots ST^{2m_1} ST^{2m_0}$ as in \eqref{eq:general-element-in-M}. We wish to relate $F(\tau)$ to $F(z)$ using the functional equations repeatedly. For this, we put $\gamma_j = T^{2m_j} ST^{2m_{j-1}}\cdots ST^{2m_0}$ for $j \geq 0$ (thus $\gamma_0= T^{2m_0}$) and write
\begin{align}
F(\tau) &= F(Mz) = (F|M)(z) j_{\theta,k}(M,z) = \left((F|M)(z)-F(z) + F(z) \right) j_{\theta,k}(M,z) \notag \\
&= \left( F(z)-\psi_M(z) \right)j_{\theta,k}(M,z) = \Big( F(z)- \sum_{j=0}^{n}{(\psi_S|\gamma_j)(z)} \Big)j_{\theta,k}(M,z). \label{eq:F-of-tau-in-terms-of-F-of-z}
\end{align}
Here, we applied cocycle property $\psi_{AB} = \psi_A|B + \psi_B$ repeatedly, combined with $\psi_{T^{2m}} = 0$, giving
\begin{equation}\label{eq:applying-cocycle-repeatedly}
\psi_M = \psi_{ST^{2m_n} \cdots ST^{2m_0}} = \sum_{j=0}^{n}{\psi_S|\gamma_j}.
\end{equation}
This can also be proved via induction. Now we use 
\[
|j_{\theta,k}(M,z)| = \left( \frac{\imag(z)}{\imag(Mz)} \right)^{k/2} = \left( \frac{\imag(z)}{\imag(\tau)} \right)^{k/2}, \quad |j_{\theta,k}(\gamma_j, z)^{-1}| = \left( \frac{\imag(\gamma_j z)}{\imag(z)}\right)^{k/2}
\]
and the triangle inequality to get
\begin{equation}\label{eq:intermediate-estimate-cocycle-1}
|F(\tau)| \leq |F(z)|\imag(z)^{k/2} \imag(\tau)^{-k/2} + \sum_{j=0}^{n}{|\psi_S(\gamma_j z)|\imag(\gamma_j z)^{k/2} \imag(z)^{-k/2} }.
\end{equation}
From our preliminary remarks on how the elements of $\mathcal{M}$ act on $\mathcal{D}$, we know that
\begin{enumerate}[(i)]
\item $|\gamma_j z| \geq 1$ for $0 \leq j \leq n$,
\item $\imag(\gamma_j z) \leq \frac{1}{2j-1}$ for $1 \leq j \leq n$ and $\imag(S \gamma_j z) \leq \frac{1}{2j+1}$ for $0 \leq j \leq n$,
\item consequently, $\imag(\tau) = \imag(Mz) = \imag(S \gamma_nz) \leq  \frac{1}{2n+1}$,
\item $\imag(z) \geq \imag(\gamma_jz) \geq \imag(S \gamma_j z) \geq \imag(\tau)$ for $0 \leq j \leq n$.
\end{enumerate}
Let us also note that, with $(c,d)$ denoting the bottom row of $M$, we have
\begin{equation}\label{eq:imag-comparision}
\imag(\tau)\imag(z) = \imag(Mz)\imag(z) = \frac{\imag(z)^2}{|c z +d|^2} = \frac{1}{c^2 +  ((\real(z)c + d)/\imag(z))^2} \leq \frac{1}{c^2} \leq 1,
\end{equation}
because $c \neq 0$, as otherwise, the element $M$ would act like a power of $T^2$ and thus not map $\mathcal{D}$ into $D(-1,1)$, contradicting what we saw above. Having assembled all of these facts, we can now estimate the terms $|\psi_S(\gamma_jz)|$ appearing in \eqref{eq:intermediate-estimate-cocycle-1}. By properties (i) and (iv) we have, for $0 \leq j \leq n$,
\begin{align*}
|\psi_S(\gamma_jz)| &=(1+r^{\beta})|e^{\pi i (\gamma_j z)r^2}- \epsilon(\gamma_j z/i)^{-k}e^{\pi i (S \gamma_j z) r^2}|\\
					&\leq 1 + |\gamma_j z|^{-k} + g(\beta)\imag(\gamma_j z)^{-\beta/2} + g(\beta)|\gamma_j z|^{-k}\imag(S \gamma_j z)^{-\beta/2} \\
					&\leq 2 + 2g(\beta) \imag(\tau)^{-\beta/2} \leq 2 \tilde{g}(\beta) \imag(\tau)^{-\beta/2},
\end{align*}
using here also  $\imag(\tau) \leq 1$. We insert this bound for $\psi_S(\gamma_jz)$ back into \eqref{eq:intermediate-estimate-cocycle-1} and obtain
\begin{align}
|F(\tau)| &\leq |F(z)|\imag(z)^{k/2} \imag(\tau)^{-k/2} + \imag(\tau)^{-k/2} \sum_{j=0}^{n}{2 \tilde{g}(\beta) \imag(\tau)^{-\beta/2}\imag(\gamma_j z)^{k/2}  } \notag\\
		  &= |F(z)|\imag(z)^{k/2} \imag(\tau)^{-k/2} +  2\tilde{g}(\beta) \imag(\tau)^{-\beta/2-k/2}\Big(\imag(z)^{k/2} +\sum_{j=1}^{n}{\imag(\gamma_j z)^{k/2}}  \Big) \notag \\
		  &\leq |F(z)|\imag(\tau)^{-k} +  2\tilde{g}(\beta) \imag(\tau)^{-\beta/2-k/2}\left( \imag(\tau)^{-k/2} + \imag(\tau)^{-1} \right), \label{eq:intermediate-estimate-cocycle-2}
\end{align}
where we used
\begin{itemize}
\item property (ii) from above to bound $\imag(\gamma_j z)^{k/2} \leq (2j-1)^{-k/2} \leq 1$ for all $j \in \{1, 2, \dots, n \}$,
\item its corollary (iii) in the form $n \leq 2n+1 \leq \imag(\tau)^{-1}$, and
\item inequality \eqref{eq:imag-comparision} (implying $\imag(z)^{k/2} \leq \imag(\tau)^{-k/2}$)
\end{itemize}
By Proposition \ref{prop:bound-F-in-fd} and since $\imag(z) \geq \imag(\tau)$, we have 
\[
|F(z)| \leq \tilde{g}(\beta)h(\beta) (1+\imag(z)^{-\beta/2-1}) \leq \tilde{g}(\beta) h(\beta) (1+\imag(\tau)^{-\beta/2-1}),
\]
where  we abbreviated by \[
h(\beta) =e^{c_1k + c_2 \beta+c_3}\Gamma(\beta/2-k+1)
\]
 the constant from \eqref{eq:upper-bound-F-in-fd}. Inserting this into \eqref{eq:intermediate-estimate-cocycle-2} we obtain $|F(\tau)| \leq 6\tilde{g}(\beta)h(\beta) \imag(\tau)^{-\beta/2-k-1}$.
From Proposition \ref{prop:bound-F-in-fd} and because $F$ is $2$-periodic, we deduce that the estimate
\begin{equation}\label{eq:bound-in-all-of-H}
|F(\tau)| \leq 6\tilde{g}(\beta)h(\beta) \left(1 + \imag(\tau)^{-\beta/2-k-1} \right)
\end{equation}
also holds when $\tau \notin D(-1,1)$, and hence in the entire upper half-plane. 
\subsubsection{Conclusion}\label{sec:conclusion}
We are now ready to prove the estimate \eqref{eq:final-bound-bkn-in-thmA} stated in part (ii) of Theorem A. We start by  recalling from \eqref{eq:def-bk-eps} that for any $y >0$, \[
(1+r^{\beta})b_{k,n}^{-\epsilon}(r) = \frac{1}{2}\int_{iy-1}^{iy+1}{(1+r^{\beta})F_k^{\epsilon}(\tau, r)e^{- \pi i n \tau}d\tau}.
\]
Hence, by the triangle inequality and by the estimate \eqref{eq:bound-in-all-of-H},
\begin{equation}\label{eq:bound-bkn-after-triangle-inequality}
|(1+r^{\beta})b_{k,n}^{-\epsilon}(r)| \leq 6\tilde{g}(\beta)h(\beta) \left( 1 + y^{-\beta/2-k-1} \right)e^{\pi n y}.
\end{equation}
If $n = 0$, we can let $y \rightarrow \infty$ and deduce
\[
\sup_{r \geq 0}{|(1+r^{\beta})b_{k,0}^{-\epsilon}(r)|} \leq 6\tilde{g}(\beta)h(\beta).
\]
For $n \geq 1$, we specialize \eqref{eq:bound-bkn-after-triangle-inequality} to $y = \tfrac{\beta}{2 \pi } \tfrac{1}{n}$ and obtain
\[
\sup_{r \geq 0}{|(1+r^{\beta})b_{k,n}^{-\epsilon}(r)|} \leq 6\tilde{g}(\beta)h(\beta)\left(1 + n^{\beta/2 + k+1}( 2\pi / \beta)^{\beta/2 + k + 1} \right)e^{\beta/2}.
\]
Recalling the definition of $h(\beta)$, this gives the bound \eqref{eq:final-bound-bkn-in-thmA} stated in part (ii) of Theorem A. 
%

\subsection{Proof of part (iii) in Theorem A}\label{sec:decay-basis-functions}

Finally, in this subsection we employ the basic methods in \cite[Section~4.1]{Ramos-Sousa-perturbation} to prove that the basis functions $b_{k,n}^{\epsilon}(r)$ satisfy additionally some 
\emph{exponential decay} with respect to $r$. That is, we prove the bound \eqref{eq:final-pointwise-bound-bkn-in-thmA} stated in part (iii) of Theorem \ref{thm:thmA} 

To start, note that an immediate consequence of the estimate \eqref{eq:final-bound-bkn-in-thmA} stated in part (ii) of Theorem A (proved in the previous subsection) together with Stirling's approximation for the Gamma function, is the estimate 
\begin{equation}\label{eq:first-bound-prior}
\sup_{r \ge 0} |(1+r^{\beta}) b_{k,n}^{\epsilon}(r)|\le C_0 C_1^{k} C_2^{\beta} (1+n)^{\beta/2 + k +1} \beta^{\beta},
\end{equation}
which is valid whenever $k \ge \frac{1}{2}, \, \beta \ge 2k+2$ and $n \ge 0.$ Here, $C_0,C_1,C_2$ are (possibly large) absolute constants that depend on none of the other parameters in the inequality. The idea is then to optimize this inequality with respect to the condition $\beta \ge 2k+2.$ We first deduce from \eqref{eq:first-bound-prior} that, for all $r \geq 1$, we have
\begin{equation}\label{eq:third-bound-prior}
|b_{k,n}^{\epsilon}(r)| \le C_0 C_1^k  (n+1)^{k+1}\exp\left(\beta \log(C_2) +(\beta/2)\log(1+n) + \beta \log (\beta) - \beta \log(r) \right). 
\end{equation}
Choose a parameter $C>2C_2$ and let $\beta_0 = \beta_0(r,n,C) = \frac{r}{C \sqrt{n+1}}$. Then, if $r \geq C \sqrt{n+1}(2k+2)$ (implying $r \geq 1$ and $\beta_0 \geq 2k+2$), a quick computation using   \eqref{eq:third-bound-prior} with $\beta = \beta_0$ shows the estimate

\begin{equation}\label{eq:final-bound-large-r}
|b_{k,n}^{\epsilon}(r)| \le C_0 C_1^k (n+1)^{k+1} \exp\left( \frac{r}{C\sqrt{n+1}} \log(C_2/C)\right) \leq  C_0 C_1^k (n+1)^{k+1} e^{-c_1 r/\sqrt{n+1}}
\end{equation}
where $c_1$ is a constant; any $c_1  \in (0, \frac{\log(2)}{2C_2}]$ is admissible, by our choice of $C$. To cover the remaining range $0 \leq r \leq C \sqrt{n+1}(2k+2)$, we sketch how to obtain bound for $\sup_{r \in \R}{|b_{k,n}^{\epsilon}(r)|}$, with an unspecified (possibly large) dependency on $k$. To do so, we estimate the integral $H_{\mu,b}(\delta)$ from Lemma \ref{lem:pinelis} in the case $b = -k/2 < 0$ as follows. We change variables $x = \delta s$ in that integral and write
\begin{align*}
|H_{\mu,b}(\delta)| & \le \delta^{\mu} \int_0^{1/(\delta e)} \frac{s^{\mu} \log^{b}(1/(\delta s))}{\sqrt{1+s^2}} \, ds \cr 
		    & = \delta^{\mu} \log^b(1/\delta) \int_0^{\delta^{-1/2}} \frac{s^{\mu} \left(1 - \frac{\log(s)}{\log(1/\delta)}\right)^b}{\sqrt{1+s^2}} \, ds + \delta^{\mu} \log^b(1/\delta) \int_{(1/\delta)^{1/2}}^{1/(\delta e)} \frac{s^{\mu} \left(1 - \frac{\log(s)}{\log(1/\delta)}\right)^b}{\sqrt{1+s^2}} \, ds\cr
		    & \le 2^{-b} \delta^{\mu} \log^b(1/\delta) \int_0^{(1/\delta)^{1/2}} \frac{s^{\mu}}{\sqrt{1+s^2}} \, ds + \delta^{\mu} \int_{\delta^{-1/2}}^{1/(\delta e)} s^{\mu - 1} \, ds. \cr
\end{align*}
From here, basic estimates imply that $|H_{\mu,b}(\delta)| \le C_{b,\mu} \delta^{\mu} \log^{b+1}(1/\delta)$, for some constant $C_{b,\mu}>0$ depending only on $b, \mu$. Repeating the same strategy as in Proposition \ref{prop:bound-F-in-fd} (compare with \eqref{eq:final-bound-F-plus}, \eqref{eq:final-bound-F-minus} and set $b = -k/2$) we obtain that 
\[
|F^{\epsilon}_k(\tau,r)|  \lesssim_{k, \epsilon} 1+\imag(\tau)^{-1}, \quad \tau \in \mathcal{D}.
\]
Also repeating the analysis in \S \ref{sec:cocycle-estimates} and \S \ref{sec:conclusion}, we obtain that 
\begin{equation}\label{eq:uniform-bad-k}
\sup_{r \geq 0}{|b_{k,n}^{\epsilon}(r)|} \lesssim_{k, \epsilon}  (1+n)^{k+1}.
\end{equation}
Gathering \eqref{eq:final-bound-large-r} and \eqref{eq:uniform-bad-k}, we obtain the bound \eqref{eq:final-pointwise-bound-bkn-in-thmA} stated in part (iii) of Theorem \ref{thm:thmA}, which we copy for convenience:
\[
|b_{k,n}^{\epsilon}(r)| \le C_k (n+1)^{k+1} e^{-c_1|r|/\sqrt{n+1}}.
\]
Here, $C_k$ is an absolute constant, depending \emph{only} on $k$, and the above estimate holds for all $n \in \N_0$ all $r \in \R$, all $k \geq 1/2$ and all $\epsilon \in \{ \pm 1\}$.
%

\begin{remark-non} The above bound is to be compared directly with \cite[Theorem~1.5]{Ramos-Sousa-perturbation}. It yields a slightly worse bound in the case $k=1/2$, in which the authors of \cite{Ramos-Sousa-perturbation} obtain a growth bound $n^{1/4} \log(1+n)^{3/2}$ on the parameter $n$, whereas our bound gives $n^{3/2}$. This is due, in part, to the bound in Lemma \ref{lem:pinelis}, which is not as sharp as the estimates in \cite[Proposition~4.4]{Ramos-Sousa-perturbation}, but also to the cocycle estimates in \S \ref{sec:cocycle-estimates}, 
which are also slightly loose in comparison to those after Proposition 4.4 in that paper. 

An improvement over these bounds is most probably possible, as well as an explicit estimate of the growth of the constants $C_k$, but a sharp growth bound seems even hard to correctly conjecture. We believe that a more careful analysis in light of the results in \cite{BRS} might lead to further improvements in our bounds (especially for small $k$), even though they quite probably do not lead to the best possible bounds. For that reason, and because we opted for a uniform, elementary treatment over all $k$,  we do not pursue that path in this manuscript. 
\end{remark-non}

\section{Perturbed interpolation formulas for radial functions}\label{sec:perturbation-radial}
In this section we prove Theorem \ref{thm:radial-perturbation-in-intro} by using parts (i) and (ii) of Theorem A. We start with some preliminaries on function spaces, which we will also use in \S \ref{sec:perturbation-sphere}. 

Let $d \ge 2.$ For every real $s \geq 1$ we define $m_s: \R^d \rightarrow [0, \infty)$ by $m_s(x) = 1+|x|^s$ (and $m_0(0):=1$). Then we define the vector space
\begin{equation}\label{eq:definition-Vs-space}
V^s(\R^d) = \left \lbrace f \in L^1(\R^d) \,:\, m_s f, m_s \hat{f} \in L^1(\R^d) \right \rbrace,
\end{equation}
and equip it with the norm
\begin{equation}\label{eq:definition-Vs-norm}
\|f\|_{V^s(\R^d)}= \|f m_s \|_{L^1(\R^d)} + \|\hat{f} m_s\|_{L^1(\R^d)}, 
\end{equation}
with respect to which $V^s(\R^d)$ is a Banach space. The space of Schwartz functions constitutes a dense subspace of $V^s(\R^d)$ and is the intersection of all Banach spaces $V^s(\R^d)$. For $f \in V^s(\R^d)$ we have that $f, \hat{f}$ admit unique representatives in $ C^s(\R^d)$ with $C^s$-norm controlled by $\|f\|_{V^s(\R^d)}$. We denote by $V_{\text{rad}}^s(\R^d)$ the subpsace of $V^s(\R^d)$ consisting of elements that are represented by measurable radial functions. It is closed and thus complete, since $V^s$-convergence implies convergence in $C^0(\R^d)$ for all $s \geq 1$. As before, $\mathcal{S}_{\text{rad}}(\R^d)$ is a dense subspace of $V_{\text{rad}}^s(\R^d)$ and is the intersection of all these spaces. In the following, we regard $d \geq 2$ as fixed and often abbreviate $V^s(\R^d)$ to $V^s$, $L^1(\R^d)$ to $L^1$, and sometimes drop the subscript ``rad", et cetera. 

Before  we turn to starting the proof of Theorem \ref{thm:radial-perturbation-in-intro}, we need one more general technical result. It proves most of the assertions about the space $V^s(\R^d)$ made above and is a generalization of a result in  \cite[\S~5.2.2]{Ramos-Sousa-uniqueness} to higher dimensions. We thank the anonymous referee for offering an improvement of an earlier version of the following proposition.  

\begin{proposition}\label{prop:decay}
Fix $d \geq 1$. For $s \geq 1$ and $f \in V^s(\R^d)$ we have $|f(x)|+|\widehat{f}(x)| \lesssim  |x|^{-\frac{s}{d+1}}$, for $|x| \geq 1$. Moreover, if $s \geq 2$, then $|\nabla f(x)|+|\nabla \widehat{f}(x)| \lesssim  |x|^{-\frac{s}{2(d+1)}}$, for $|x| \geq 1$.
\end{proposition}
\begin{proof} 
Fix a nonzero $f \in V^s(\R^d)$. Notice first that, by Fourier inversion, as $s \ge 1,$ one has $|f(x)|, |\nabla f(x)| \le \|f\|_{V^s}$ for all $x \in \R^d$. Now fix $x \in \R^d$ with $|x| \geq 1$. Define $r_x := |f(x)|/(2\|f\|_{V^s}) \leq 1/2$ and let $B_x := B_{r_x}(x)$ denote the ball centered at $x$ with radius $r_x$.  For any $y  \in B_x,$ we have the estimate
\[
|f(y) - f(x)| \le \left( \sup_{z \in B_x} |\nabla f(z)|\right) |x-y| \le \frac{|f(x)|}{2\|f\|_{V^s}} \|f\|_{V^s} = \frac{|f(x)|}{2},
\]
which implies $|f(y)| \ge \frac{|f(x)|}{2}$. Moreover, for any $y \in B_x$, we have $|y|  \geq |x|- r_x \geq |x|-1/2 \geq |x|/2$. These inequalities imply that
\[
\|m_s f\|_{L^1(\R^d)} \ge \int_{B_x} (1+|y|^s)|f(y)| \, dy  \geq |B_x|(1+2^{-s}|x|^s)|f(x)|/2 \gtrsim_{s,f} |B_x| |f(x)||x|^s.
\]
Since $|B_x| = \alpha_d r_x^d = \beta_{d,s,f} |f(x)|^d$ for some constants $\alpha_d$, $\beta_{d,s,f}$ with the indicated dependencies,  we obtain that $|f(x)| \lesssim_{d,s,f} |x|^{-\frac{s}{d+1}}$. This proves the result for $f$ and, by symmetry, one obtains the same estimate for $\widehat{f}.$ The second assertion then follows from the first via an argument using Taylor's theorem with remainder (see \cite[Lemma 6.1]{Stoller-fip-spheres} for a sketch).  
\end{proof}
We now turn to the definition of the operator that is the perturbation of the identity, expressed via the interpolation formula \eqref{eq:interpolation-formula-radial-in-thmA} of part (i) in Theorem A. 

Consider two sequences of real numbers $\varepsilon_n, \hat{\varepsilon}_n$, indexed by $n \in \N_0$. We will eventually assume that these are sufficiently small. For now, we only assume that $|\varepsilon_n|,|\hat{\varepsilon}_n| \leq 1$ for all $n \geq 1$. We want to define a bounded linear operator $T: V_{\text{rad}}^s \rightarrow V_{\text{rad}}^s$ by\footnote{We suppress dependence on $d$, $\varepsilon_n$ and $\hat{\varepsilon}_n$ from the notation for this operator.}
\begin{equation}\label{eq:definition-T-radial}
Tf = f(\varepsilon_0)a_{d/2,0} + \sum_{n=1}^{\infty}{ f(\sqrt{n+\varepsilon_n}) a_{d/2,n} } + \hat{f}(\hat{\varepsilon}_n)\tilde{a}_{d/2,0} + \sum_{n=1}^{\infty}{ \hat{f}(\sqrt{n+\hat{\varepsilon}_n})\tilde{a}_{d/2,n}}
\end{equation}
and prove that it is close to the identity. Here, the Schwartz functions $a_{d/2,n},\tilde{a}_{d/2,n}$ are those in \eqref{eq:definition-an} and are viewed as elements of $\mathcal{S}_{\text{rad}}(\R^d) \subset V_{\text{rad}}^s$. It is not clear whether and how formula \eqref{eq:definition-T-radial}  defines a bounded linear operator on some $V_{\text{rad}}^s$, but it is clear that the formula \eqref{eq:definition-T-radial} defines at least a continuous linear map $T : \mathcal{S}_{\text{rad}}(\R^d) \rightarrow \mathcal{S}_{\text{rad}}(\R^d)$, because any Schwartz seminorm of the functions $a_{d/2,n}$, $\tilde{a}_{d/2,n}$ grows at most polynomially with $n$. We will show that $T$ \emph{extends} continuously to a bounded operator on $V_{\text{rad}}^s$, by showing that $f \mapsto Tf -f$ is continuous as a linear map from $\mathcal{S}_{\text{rad}}(\R^d)$ to $V^s$, with respect to the  $V^s$-norm on target and $V^1$-norm (hence $V^s$-norm) on the source provided $\varepsilon_n, \hat{\varepsilon}_n$ are assumed sufficiently small. 

For $f \in \mathcal{S}_{\text{rad}}(\R^d)$  the interpolation formula \eqref{eq:interpolation-formula-radial-in-thmA} holds, so to estimate $\|Tf-f\|_{V^s}$ in terms of $\|f\|_{V^1}$ it suffices to estimate, for $n \geq 1$,
\begin{align*}
\|\left(f(\sqrt{n + \varepsilon_n})-f(\sqrt{n}) \right) a_{d/2,n}\|_{V^s} \lesssim |\sqrt{n + \varepsilon_n}-\sqrt{n} |  \; \|f\|_{V^1} \|a_{d/2,n}\|_{V^s}  \leq  \frac{|\varepsilon_n|}{\sqrt{n}}\, \|f\|_{V^s} \,\|a_{d/2,n}\|_{V^s}
\end{align*}
and similarly $\|\hat{f}(\sqrt{n + \hat{\varepsilon}_n})-\hat{f}(\sqrt{n})\|_{V^s}$.
Here, we implicitly used that $\tilde{a}_{d/2,n}$ is the Fourier transform of $a_{d/2,n}$ on $\R^d$. The necessary modifications for $n = 0$ are clear. Now we bound the $V^s$-norms
\[
\|a_{d/2,n}\|_{V^s} = \|\tilde{a}_{d/2,n}\|_{V^s} \leq \|b_{d/2,n}^{+}\|_{V^s} +  \|b_{d/2,n}^{-}\|_{V^s} \leq 2 \|b_{d/2,n}^{+} m_s \|_{L^1} + 2 \|b_{d/2,n}^{-} m_s \|_{L^1}.
\]
For $\beta >0$ to be determined, we can write
\begin{align*}
\|b_{d/2,n}^{\epsilon} m_s \|_{L^1(\R^d)} &= \text{area}(S^{d-1}) \int_{0}^{\infty}{|b_{d/2,n}^{\epsilon}(r)| (1+r^{s})r^{d-1} dr} \\
&\leq \text{area}(S^{d-1}) \sup_{r \geq 0}{|b_{d/2,n}^{\epsilon}(r)(1+r^{\beta})|}\int_{0}^{\infty}{\frac{1+t^s}{1+t^{\beta}}t^{d-1}dt}.
\end{align*}
In order to be able to use part (ii) of Theorem A and to make the last integral convergent, we take \[
\beta = d +s+1 \geq 2(d/2) + 2,
\]
so that for all $n \geq 0$,
\begin{equation}\label{eq:vs-norm-an}
\|a_{d/2,n}\|_{V^s} = \|\tilde{a}_{d/2,n}\|_{V^s} = M(s,d)(1+n)^{d+(s/2)+(3/2)},
\end{equation}
where $M(s,d)$ is a constant whose dependence on $s$ and $d$ may be read off from part (ii) of Theorem A, together with the the dependence on $d$ on the surface area measure of $S^{d-1}$.
%
Returning to the original estimate, we now have
\begin{equation}\label{eq:close-to-the-identity-estimate}
\|Tf-f\|_{V^s} \leq \|f\|_{V^1} M(s,d) \left((|\varepsilon_0|+|\hat{\varepsilon}_0|) + \sum_{n=1}^{\infty}{\frac{|\varepsilon_n|+|\hat{\varepsilon}_n|}{\sqrt{n}}(1+n)^{d+(s/2)+(3/2)}} \right).
\end{equation}
Note here that $\|f\|_{V^1} \lesssim \|f\|_{V^s}$ since $s \geq 1$. We see that, by assuming $\varepsilon_n, \hat{\varepsilon}_n$ to be sufficiently small, we can ensure that the above series converges and we can make its value as small as we like, in particular $< 1$.

It has been noted to us that the approach in the following proof is similar in flavour to a well-known approach in time-frequency analysis, namely the one in Corollary 5.1.3 in \cite{Groechenig}.

\begin{proof}[Proof of Theorem \ref{thm:radial-perturbation-in-intro}] Retain the above set up and assume the smallness condition \eqref{eq:smallness-assumption}, that is, assume that for some $\eta >0$ and $\delta >0$, we have $|\varepsilon_n|+|\hat{\varepsilon}_n| \leq \delta (1+n)^{-d-(s/2)-2-\eta}$ for all $n \geq 0$. 
It follows immediately from the estimate \eqref{eq:close-to-the-identity-estimate}, that the operator $T$ extends to a bounded operator from $V^s_{\text{rad}}(\R^d)$ to itself.  We can take $\delta$ small enough in terms of $s, \eta, d$ and the constant $M(s,d)$ in \eqref{eq:close-to-the-identity-estimate}, to ensure that $T$ is invertible. Having done so, we now prove the interpolation formula \eqref{eq:pertrubed-radial-interpolation-formula} for the functions $h_{d,n} := T^{-1}(a_{d/2,n}), \tilde{h}_{d,n} := T^{-1}(\tilde{a}_{d/2,n})$. Note that these are indeed zero if $n < \nu_{-}(d/2)$, because $a_{d/2,n}$ and $\tilde{a}_{d/2,n}$ are.

So let  $s' \geq (d+1)(s+2d+5 + 2\eta)$ be an integer and let  $f \in V_{\text{rad}}^{s'}$ be arbitrary. Then $f \in V_{\text{rad}}^s$ and hence $f = T^{-1}(Tf)$ since $T$ is invertible on $V_{\text{rad}}^s$. Since, by Proposition \ref{prop:decay}, $f(\sqrt{n+\varepsilon_n})$ and $\hat{f}(\sqrt{n+\hat{\varepsilon}_n})$ are both $O((1+n)^{-s'/(2(d+1))})$  as $n \rightarrow \infty$, and since the $V^s$-norms of $a_{d/2,n}$ and $\tilde{a}_{d/2,n}$ are both $O((1+n)^{d+(s/2)+(3/2)})$ by \eqref{eq:vs-norm-an}, the series \eqref{eq:definition-T-radial} converges absolutely in the $V_{\text{rad}}^s$-norm and hence equals $Tf$. Since $T^{-1}$ is continuous on $V_{\text{rad}}^s$, we may apply $T^{-1}$ to this series by applying it to each summand and thus finish the proof. 
\end{proof}

\section{Fourier uniqueness with perturbed spheres}\label{sec:perturbation-sphere}
This section is devoted to the proof of Theorem \ref{thm:non-radial-uniqueness-intro}, pertaining to not necessarily radial functions on $\R^d$, where $d \geq 2$. We start by applying some of the results proved in \cite[\S 2]{Stoller-fip-spheres}, that allow us deduce an interpolation result for non-radial Schwartz functions on $\R^d$ from corresponding results for radial functions in dimensions $d, d+2, d+4, \dots$, such as the ones in part (i) of Theorem A. 

For half-integers $k \geq 1$ let $a_{k,n}, \tilde{a}_{k,n} \in \mathcal{S}_{\text{even}}(\R)$ be defined via \eqref{eq:definition-an} and \eqref{eq:def-bk-eps}. Recall that they are both zero if $n < \nu_{-}(k) =  \lfloor (k+2)/4 \rfloor$ and that, when viewed as a radial function on $\R^{2k}$, the function $\tilde{a}_{k,n}$ is the Fourier transform of $a_{k,n}$. By the radial interpolation formulas \eqref{eq:interpolation-formula-radial-in-thmA} of part (i) in Theorem A and by \cite[Corollary~2.1]{Stoller-fip-spheres}, we have for all $f \in \mathcal{S}(\R^d)$ and all $x \in \R^d$,
\begin{align}
f(x) &= \sum_{m=0}^{\infty}{ \Big( \sum_{n=\nu_{-}(d+2m)}^{\infty}{ a_{d/2+m,n}(|x|) \frac{1}{\sqrt{n}^m} \int_{S}{f(\sqrt{n}\zeta) Z_m^d(x, \zeta) d\zeta  } }} \notag \\
 &\qquad \qquad+   \sum_{n=\nu_{-}(d+2m)}^{\infty}{ i^m\tilde{a}_{d/2+m,n}(|x|)\frac{1}{\sqrt{n}^m} \int_{S}{\hat{f}(\sqrt{n}\zeta) Z_m^d(x, \zeta) d\zeta  }} \Big) \label{eq:double-series-formula},
\end{align}
where $d\zeta$ denotes integration with respect to probability surface measure on the unit sphere $S = S^{d-1}$ (we shall switch between these two notations for the unit sphere throughout the rest of the text, omitting thus the superscript sometimes). Here, $Z_m^d(x,y)$ denotes the reproducing kernel on the space of spherical harmonics of degree $m$ and, if $n =0$, the definition of $ \frac{1}{\sqrt{n}^m} \int_{S}{f(\sqrt{n}\zeta) Z_m^d(x, \zeta) d\zeta  }$ is 
\[\lim_{r \rightarrow 0} \frac{1}{r^{m}} \int_{S}{f(r\zeta) Z_m^d(x, \zeta) d\zeta} = \sum_{|\alpha| =m}{\frac{\partial^{\alpha}f (0)}{\alpha!} \int_{S}{Z_m^d(x, \zeta) \zeta^{\alpha}d\zeta}},
\]
the last sum taken over all $\alpha \in \N_0^m$ satisfying $|\alpha| = m$ (see also \cite[Proposition 2.1]{Stoller-fip-spheres}). The series in \eqref{eq:double-series-formula} converges in the sense that $\sum_{m =0}^{\infty}{|(\cdots)|} < \infty$, but the double sum over $n$ and $m$ might not converge absolutely. Thus, a formal, unjustified manipulation of \eqref{eq:double-series-formula} suggests that the following interpolation formula could hold, for every $f \in \mathcal{S}(\R^d)$ and every $x \in \R^d$:
\begin{align}
f(x) &\stackrel{?}{=} a_{d/2,0}(|x|)f(0) + \sum_{n=1}^{\infty}{\int_{S}{K_n^d(x, \zeta)f(\sqrt{n}\zeta) d\zeta}} 
       \notag \\
       &\qquad + \tilde{a}_{d/2,0}(|x|)\hat{f}(0) +\sum_{n=1}^{\infty}{\int_{S}{\tilde{K}_n^d(x, \zeta)\hat{f}(\sqrt{n}\zeta) d\zeta}}, \label{eq:hypothetical-non-radial-interpolation-formula}
\end{align}
where, for every $n \geq 1$, we define
\begin{align}
K_n^d(x, y) = \sum_{m = 0}^{4n+1}{a_{d/2+m,n}(|x|) n^{-m/2}Z_m^d(x, y)}, \label{eq:def-A_n}\\
\tilde{K}_n^d(x, y) = \sum_{m = 0}^{4n+1}{i^m\tilde{a}_{d/2+m,n}(|x|) n^{-m/2}Z_m^d(x, y)}. \label{eq:def-tilde-A_n}
\end{align}
Note here that if $d \geq 4$, then $a_{d/2,0} = 0 = \tilde{a}_{d/2,0}$ and that the kernels \eqref{eq:def-A_n} and \eqref{eq:def-tilde-A_n} would be unchanged if we summed over from $m = 0$ to infinity, because if $m > 4n+1$, then
\[
n < \frac{m-1}{4} = \frac{m+(d/2)+2}{4}- \frac{3+(d/2)}{4} \leq \frac{m+(d/2)+2}{4} -1 \leq \nu_{-}(m+(d/2))
\]
and hence $a_{d/2+m,n} = 0 = \tilde{a}_{d/2+m,n}$. Similarly, we see that $K_n^d = 0 = \tilde{K}_n^d$ for $n < \nu_{-}(d/2) = \lfloor (d + 4)/8 \rfloor$ since in that case, we have $a_{d/2+m,n} = 0 = \tilde{a}_{d/2+m,n}$ for all $m \geq 0$. 


Even though we did not justify the derivation of the hypothetical formula \eqref{eq:hypothetical-non-radial-interpolation-formula} from \eqref{eq:double-series-formula}, we note that each integral in it is perfectly well-defined for any continuous and integrable function $f: \R^d \rightarrow \C$ and would remain perfectly well-defined if we replaced $\sqrt{n}$ in $f(\sqrt{n}\zeta)$ by $f((\sqrt{n}+ \varepsilon_n(\zeta)) \zeta)$ for any continuous (or bounded measurable) function $\varepsilon_n : S \rightarrow \R$, because of the definition of $K_n^d$ and $\tilde{K}_n^{d}$ as \emph{finite} sums. In fact, for each fixed $r > 0$,  the function $\xi \mapsto \int_{S}{K_n^d(\xi, \zeta)f(r\zeta)d \zeta}$  belongs to $\mathcal{S}(\R^d)$ and it is the Fourier transform of $x \mapsto \int_{S}{ \tilde{K}_n^d(x, \zeta)f(r \zeta)d \zeta}$.
 

We also remark that,  by the orthogonality relations of spherical harmonics and the Hecke-Funk formula, \eqref{eq:hypothetical-non-radial-interpolation-formula} is valid whenever the Schwartz function $f$ is equal to a radial Gaussian multiplied by a polynomial.  

We now turn to the perturbation idea and slowly work up to the proof of Theorem \ref{thm:non-radial-uniqueness-intro}. Consider two sequences of measurable functions $\varepsilon_n, \hat{\varepsilon}_n : S^{d-1} \rightarrow \R$, $n \geq 1$, two vectors $\varepsilon_0, \hat{\varepsilon}_0 \in \R^d$ and put
\begin{equation}\label{eq:sup-of-epsilon-functions}
\sigma_n = \sup_{\zeta \in S^{d-1}}{|\varepsilon_n(\zeta)|} + \sup_{\zeta \in S^{d-1}}{|\hat{\varepsilon}_n(\zeta)|}, \qquad \sigma_0 = |\varepsilon_0| +|\hat{\varepsilon}_0|.
\end{equation}
Recall the definition of the Banach spaces $V^s = V^s(\R^d)$ in \eqref{eq:definition-Vs-space}. We will work with a general parameter $s \geq 1$ for a while and later specialize to $s = 1$. Our goal is to show that
\begin{align}\label{eq:def-Tf-first-line}
Tf(x)= f(x)-  a_{d/2,0}(|x|)\left[f(0)-f(\varepsilon_0)\right]- \sum_{n=1}^{\infty}{\int_{S}{K_n^d(x, \zeta)\left[ f(\sqrt{n}\zeta)-f(\sqrt{n}\zeta+ \varepsilon_n(\zeta)\zeta) \right]d\zeta}} \\
\quad \qquad - \tilde{a}_{d/2,0}(|x|)\left[\hat{f}(0)- \hat{f}(\hat{\varepsilon}_0)\right] -\sum_{n=1}^{\infty}{\int_{S}{\tilde{K}_n^d(x, \zeta)\left[ \hat{f}(\sqrt{n}\zeta)-\hat{f}(\sqrt{n}\zeta+ \hat{\varepsilon}_n(\zeta)\zeta) \right]d\zeta}} \label{eq:def-Tf-second-line}
\end{align}
defines a \emph{bounded} linear operator $T : V^s \rightarrow V^s$ satisfying $\|\id_{V^s} -T\| < 1$, provided the quantities $\sigma_n$ are sufficiently small.  We can achieve this goal similarly to what we did in the radial setting in \S \ref{sec:perturbation-radial}. Namely we simply use $|f(0)-f(\varepsilon_0)| \lesssim \|f\|_{V^s} |\varepsilon_0|$ and
\[
 \left \| \int_{S}{K_n^d( \cdot , \zeta)\left[ f(\sqrt{n}\zeta)-f(\sqrt{n}\zeta+ \varepsilon_n(\zeta)\zeta) \right]d\zeta}\right\|_{V^s} \lesssim \sigma_n \|f\|_{V^1} \sup_{\zeta \in S}{\|K_n^d( \cdot , \zeta)\|_{V^s}}
\]
and the same for the terms involving $\hat{f}$. In both estimates, the implied constants are absolute. To bound the $V^s$-norms, we compute, for fixed $\zeta \in S^{d-1}$, using polar coordinates
\begin{align}
\|K_n^d( \cdot , \zeta)\|_{V^s} &\leq \sum_{m=0}^{4n+1}{n^{-m/2} \int_{\R^d}{|a_{d/2+m,n}(|x|)| |Z_m^d(x, \zeta)| (1+|x|^s) dx}} \label{eq:V^s-norm-estimate-An-first-line}\\
&= \sum_{m=0}^{4n+1}{n^{-m/2} \int_{0}^{\infty}{\int_{S^{d-1}}|a_{d/2+m,n}(r)| r^{m} |Z_m^d(\omega, \zeta)| (1+r^s)d \omega\, r^{d-1} dr}}  \\
& \lesssim_d \sum_{m=0}^{4n+1}{n^{-m/2} \dim{( \mathcal{H}_m(\R^d))} \int_{0}^{\infty}{|a_{d/2+m,n}(r)|  (1+r^s)r^{m+ d-1} dr}} \label{eq:V^s-norm-estimate-An},
\end{align} 
where we bounded the  $L^1$-norm of $\omega \mapsto Z_m^d(\omega, \zeta)$ by its $L^2$-norm, which is independent of $\zeta$ and bounded by the indicated dimension. For every index $m$ in the above sum we write, for some $\beta_m >0$ to be chosen,
\[
\int_{0}^{\infty}{|a_{d/2+m,n}(r)|  (1+r^s)r^{m+ d-1} dr} \leq \sup_{r \geq 0}{|a_{d/2+m,n}(r)|(1+r^{\beta_m})} \int_{0}^{\infty}{\frac{1+t^s}{1+t^{\beta_m}}t^{m+d-1}dt}.
\]
For the integral to converge we want $s-\beta_m + m+d-1 <-1$, but we also want $\beta_m \geq 2(d/2+m) +2$ so that we can apply the estimate \eqref{eq:final-bound-bkn-in-thmA} in part (ii) of Theorem A. Thus, we choose $\beta_m =2(d/2+m) + (1+s)$ in which case the integrand in the $t$-integral is $O(t^{-2-m})$ as $t \rightarrow \infty$ and can thus bounded independently of $m$. Thanks to the bound \eqref{eq:final-bound-bkn-in-thmA} in part (ii) of Theorem A, we have that $\sup_{r \geq 0}{|a_{d/2+m,n}(r)|(1+r^{\beta_m})}$ is bounded by an  absolute constant times
\begin{equation}\label{eq:very-bad-bound}
  n^{d+2m+\frac{s+3}{2}} \Gamma(\tfrac{s+3}{2}) \left( 1+ \left[(d+2m +1+s)/(2 \pi e) \right]^{ \frac{d+2m+1+s}{2}}  \right) e^{c_1 m + c_2 s + c_3 d + c_4},
\end{equation}
where $c_1, c_2, c_3, c_4$ are absolute constants. To simplify a bit, we set $s = 1$ from now on. Returning to  \eqref{eq:V^s-norm-estimate-An} and using $\dim{\mathcal{H}_m(\R^d)} \lesssim_d (1+m)^{d-2}$, we get
\begin{align}
\sup_{\zeta  \in S^{d-1}}{\|K_n^d( \cdot , \zeta)\|_{V^1}} & \lesssim_{d}  \sum_{m=0}^{4n+1}{n^{-m/2}(1+m)^{d-2}n^{2m+ d+ 2}(2m +d+2)^{m+(d/2)+1}e^{c_1 m} } \notag \\
&\leq (4n+2)(4n+3)^{d-2} n^{d+2 + (3/2)(4n+2)}(d+4 + 8n)^{4n+(d/2) + 2}e^{c_1(4n+1)} \notag \\
& \lesssim_d n^{10n + (5/2)d + c_5}, \label{eq:final-estimate-V1-norm-of-An}
\end{align}
for some absolute constant $c_5>0$. Note that the same bound holds for $\tilde{K}_n^d$. Thus, if we impose
\begin{equation}\label{eq:smallness-condition-on-sigma-n}
\sigma_n  \leq \delta (1+n)^{-10n - (5/2)d-c_5-1.1}, \qquad \delta = \delta_d > 0,
\end{equation}
then the operator $T$ is bounded on $V^1(\R^d)$ and also invertible if we moreover assume that $\delta$ is sufficiently small in terms of the implied constant in \eqref{eq:final-estimate-V1-norm-of-An}. We are now ready to give the proof of Theorem \ref{thm:non-radial-uniqueness-intro}. 
\begin{proof}[Proof of Theorem \ref{thm:non-radial-uniqueness-intro}]\label{sec:proof-of-thm-non-radial}
Retain the above notations and assume that \eqref{eq:smallness-condition-on-sigma-n} holds with $\delta$ so small that the operator $T$ is invertible on $V^1(\R^d)$.  Let $f \in \mathcal{S}(\R^d)$ be such that 
\begin{equation}\label{eq:vanishing-condition-non-radial-f}
f(\sqrt{n}\zeta + \varepsilon_n(\zeta) \zeta) = 0 = \hat{f}(\sqrt{n}\zeta + \hat{\varepsilon}_n(\zeta)\zeta)
\end{equation}
for all $n \geq n_0(d) =  \lfloor (d+4)/8 \rfloor$ (with the meaning of $f(\varepsilon_0) = 0= \hat{f}(\hat{\varepsilon}_0)$ for $n = 0$). Our goal is to show that $f = 0$. Since $T$ is invertible and $f \in V^1(\R^d)$, this is the same as showing $Tf = 0$, equivalently $f-Tf = f$. Note that under the vanishing assumption \eqref{eq:vanishing-condition-non-radial-f}, the series defining $f-Tf$ is equal to the right hand side of \eqref{eq:hypothetical-non-radial-interpolation-formula}, so it will suffice to show that that series equals $f$.

Writing $f(\sqrt{n}\zeta) = f(\sqrt{n}\zeta)-f(\sqrt{n}\zeta +\varepsilon_n(\zeta)\zeta)$ and applying Fourier inversion (or alternatively the mean-value theorem), we see that $f$, and  similarly $\hat{f}$, decay incredibly fast along the spheres $\sqrt{n}S^{d-1}$. Precisely, we have, by \eqref{eq:smallness-condition-on-sigma-n},
\[
\sup_{ \sqrt{n} S^{d-1}}{|f|} + \sup_{ \sqrt{n} S^{d-1}}{|\hat{f}|} \lesssim_{f} \sigma_n \leq \delta (1+n)^{-10n - (5/2)d-c_5-1.1}.
\]
The $V^1$-norm controls the $L^{\infty}$-norm, so the initial step (the application of the triangle inequality in \eqref{eq:V^s-norm-estimate-An-first-line}) in our derivation  of the estimate \eqref{eq:final-estimate-V1-norm-of-An} shows that the double series \eqref{eq:double-series-formula} converges absolutely and hence the derivation (interchange of sums an integrals) of \eqref{eq:hypothetical-non-radial-interpolation-formula} from \eqref{eq:double-series-formula} is justified for this particular function $f$. Thus $f-Tf =f$ and hence $f = 0$ as desired.
\end{proof}

\begin{remark}\label{rmk:discrete-uniqueness-sets}
One notices, from the proof above, that the perturbations we chose in the statement of Theorem \ref{thm:non-radial-uniqueness-intro} have been taken with the special property that they are \emph{parallel} to the vector of the sphere being perturbed. That is, we have considered perturbations of the form $\sqrt{n}\zeta + \varepsilon_n(\zeta)\zeta,$ with $\zeta \in S^{d-1}.$ This has been adopted as one of the primary goals of this work was to prove that one may deform the spheres $\sqrt{n}S^{d-1}$ slightly into \emph{ellipsoids}, which is one of the consequences of Theorem \ref{thm:non-radial-uniqueness-intro}. 
	
One sees, however, that the same statement holds -- with almost exactly the same proof -- if one replaces the perturbations $\varepsilon_n(\zeta)\zeta$ in the above proof by more general vector-valued perturbations  $\varepsilon_n(\zeta) \in \R^d$ with \emph{any} measurable choice of direction $\frac{\varepsilon_n(\zeta)}{|\varepsilon_n(\zeta)|} \in S^{d-1},$ as long as $\sup_{\zeta \in S^{d-1}, n \ge 1}{ |\varepsilon_n(\zeta)| n^{An + B}}$ is sufficiently small, for some constants $A,B>0$ (depending only on $d$). This may be used, as we shall see, to obtain some \emph{new discrete Fourier uniqueness sets.} 

Indeed, fix, for each $n \geq 1$, two partitions
\begin{equation}\label{eq:partition-spheres}
	\sqrt{n}S^{d-1} = \bigsqcup_{i \in I(n)}{ \sqrt{n}\Omega_{n,i}}, \qquad \sqrt{n}S^{d-1} = \bigsqcup_{j \in J(n)}{\sqrt{n} \tilde{\Omega}_{n,j}},
\end{equation}
where $\Omega_{n,i}, \tilde{\Omega}_{m,j} \subset S^{d-1}$ are  nonempty measurable subsets and $I(n), J(n)$ are finite sets of indices. Let $\zeta_0, \tilde{\zeta}_0 \in \R^d$ and
$\zeta_{n,i} \in \Omega_{n,i}, \tilde{\zeta}_{n,j}  \in \tilde{\Omega}_{n,j}$
be arbitrary, and suppose that $|\zeta_0| + |\tilde{\zeta}_0| \leq \delta$ and that these partitions satisfy 
\[
\max{\{\diam(\sqrt{n} \Omega_{n,i}), \diam(\sqrt{n}\tilde{\Omega}_{n,j})\}} \leq \mu n^{-An - B} \quad \text{ for all } n \in \N\text{ and } (i,j) \in I(n) \times J(n),
\]
for $\mu, \delta > 0$ sufficiently small, $A,B > 0$ sufficiently large (in terms of $d$ only). Then the pair of subsets 
\begin{equation}\label{eq:uniqueness-pair}
	\left( \{\zeta_0\} \cup \{\sqrt{n}\zeta_{n,j}\}_{ n \in \N, i \in I(n)},  \{ \tilde{\zeta}_0\} \cup \{ \sqrt{n}\tilde{\zeta}_{n,j}\}_{ n \in \N, j \in J(n)} \right)
\end{equation}
is a Fourier uniqueness pair for $\Ss(\R^d)$. This can be proved by defining the measurable functions $\varepsilon_n, \tilde{\varepsilon}_n : S^{d-1} \rightarrow \R^d$ by 
\[
\varepsilon_n(\zeta)  = \sqrt{n}\left(\sum_{i \in I(n)}{\chi_{\Omega_{n,i}}(\zeta)\zeta_{n,i} }\right) - \sqrt{n}\zeta, \qquad \tilde{\varepsilon}_n(\zeta)  = \sqrt{n} \left(\sum_{j \in J(n)}{ \chi_{\tilde{\Omega}_{n,j}}(\zeta) \tilde{\zeta}_{n,j}   }\right) - \sqrt{n}\zeta.
\]
By replacing the ``old" radial perturbations $\varepsilon_n(\zeta)\zeta$ in \eqref{eq:def-Tf-first-line} by the ``new" perturbations $\varepsilon_n(\zeta)$,  defined above and by making similar adjustments elsewhere in the proof, we obtain the following result.

%

\begin{corollary}\label{cor:discrete-uniqueness-sets}
There exist \emph{discrete} sets $A, B \subset \R^d$ forming a Fourier uniqueness pair for $\Ss(\R^d)$,  all of whose points lie on $\cup_{n \geq 1}{\sqrt{n}S^{d-1}}$ (except possibly for one point close to the origin), and such that the numbers $|\sqrt{n}S^{d-1} \cap A|$ and $|\sqrt{n}S^{d-1} \cap B|$ are $O_d(n^{\alpha_d n + \beta_d})$ for some $\alpha_d, \beta_d >0$. 
\end{corollary}

This is, in particular, directly related to the recent work of the second author with D. Radchenko \cite{RadchenkoStoller}, where they exhibited many examples of discrete Fourier \emph{non-uniqueness} subsets $A,B$ of $\cup_{n \ge 0} \sqrt{n} S^{d-1},$ where $|\sqrt{n}S^{d-1} \cap A|$ and $|\sqrt{n}S^{d-1} \cap B|$ are of the order of $c \cdot n^{d-1}$ as $n \to \infty.$ Corollary \ref{cor:discrete-uniqueness-sets} may be then regarded as a positive counterpart of their main result. 
\end{remark}

\section{An application to Heisenberg uniqueness pairs}\label{sec:application}

We will prove Theorem \ref{thm:application-hup} as a consequence of Theorem \ref{thm:radial-perturbation-in-intro} using the main strategy in \cite{HMRV}. Consider a (large) parameter $s_0 \geq 1$ and an \emph{odd} function $f \in V^{s_0}(\R)$. Assume that the Fourier transform $\widehat{\mu_f}$, of the measure $\mu_f$ defined as in \eqref{eq:definition-of-mu_f}, vanishes on a perturbed lattice cross given as in \eqref{eq:perturbed-lattice-cross} with $\varepsilon_0 = \hat{\varepsilon}_0$ and $\alpha = \beta =1$.  We want to prove that if $s_0$ is sufficiently large and the $\varepsilon_n$, $\hat{\varepsilon}_n$ are sufficiently small, then $f = 0$.  Consider the function
\[
g(t) = \alpha(t)f(t), \quad \text{where} \quad  \alpha(t) =  t^3 \sqrt{1 +t^{-4}} = t \sqrt{1+t^4},
\]
which appeared in the integral characterizing $\mu_f$ in \eqref{eq:definition-of-mu_f}. Since $f$ is odd, $g$ is even and we clearly have $g(0) = 0 = g'(0)$. An elementary calculations shows for that for each integer $j \geq 0$, there is a polynomial $P_j$ of degree at most $3j  +1$, so that $\alpha^{(j)}(t) = P_j(t) (1+t^4)^\frac{1-2j}{2}$. In particular, we have $\alpha^{(j)}(t) = O(|t|^{3-j}) = O(|t|^3)$ as $|t| \rightarrow \infty$. Combining this observation with Proposition \ref{prop:decay}, we see that $g \in V^{s_1}(\R)$ where $s_1 = s_1(s_0) \rightarrow \infty$ as $s_0 \rightarrow \infty$.

Next, we define the function $G : \R \rightarrow \C$ by $G(x) := \int_{- \infty}^{x}{g(t) dt}$. We recall that, as part of our assumption,  we have 
\[
0 = \widehat{\mu_f}(\varepsilon_0, \hat{\varepsilon}_0) = \widehat{\mu_f}(0,0) = \int_{\R}{g(t)dt} = 0.
\]
Therefore, $ G(x) = -\int_{x}^{\infty}{g(t)dt}$. Moreover, $G'(x)= g(x)$ and $G(0) = 0$ since $g$ is even. In fact, since $g(0) = 0=g'(0)$, we have $G(0) = G'(0) = G''(0) =0$.  Also,  $G \in V^{s_2}(\R)$,  where $s_2 = s_2(s_1) \rightarrow \infty$ as $s_1 \rightarrow \infty$. To see this, note that we have, for $1 < n \leq s_1$ and $x> 0$, the simple estimate
\[
|G(x)| = \left| -\int_{x}^{\infty}{g(t)t^{n}t^{-n}dt}\right|  \lesssim_{n,g}\int_{x}^{\infty}{t^{-n}dt} = \frac{1}{n-1}x^{-(n-1)}
\]
and similarly for $x<0$. Analogously, we have that $\partial_x G = g.$ As we have seen that $G \in L^2,$ we may take Fourier transforms on both sides of this identity. Thus, we have $\widehat{G}(\xi) \cdot \xi^k = \frac{\xi^{k-1}}{2 \pi} \widehat{g}(\xi).$ As $g \in V^{s_1}(\R),$ this plainly implies that $G \in V^{s_2}(\R),$ with $s_2 \to \infty$ as $s_1 \to \infty,$ as desired. 

Finally, we consider the function $\Phi : \R^4 \rightarrow \C$, defined by 
\begin{equation}\label{eq:definition-Phi}
\Phi(x) = \int_{\R}{G(\tau) e^{\pi i \tau |x|^2}d \tau}, \quad x \in \R^4.
\end{equation}
This transformation of $G$ was first introduced and studied in \cite{HMRV} in this framework. So defined, $\Phi$ is obviously radial, and we claim that $\Phi \in V_{\text{rad}}^{s_3}(\R^4)$, where $s_3 =s_3(s_2) \rightarrow \infty$, as $s_2 \rightarrow \infty$. To prove this claim, note that $\Phi(x) = \mathcal{F}_{1}(G)(-|x|^2/2)$, where $\mathcal{F}_{1}(\phi)(\xi) = \int_{\R}{\phi(\tau) e^{-2 \pi i \tau \xi }dt}$ denotes the one-dimensional Fourier transform of a function $\phi : \R \rightarrow \C$. A simple computation using polar coordinates and a quadratic change of variables then shows that $\Phi m_{s_3} \in L^1(\R^4)$ as long as $s_3 \geq 2s_2 + 4$. To show that also $\widehat{\Phi}m_{s_3} \in L^1(\R^4)$, it suffices to show that sufficiently many partial derivatives of $\Phi$ are in $L^1(\R^4)$, which is readily seen via differentiation  under the integral sign and by applying Proposition \ref{prop:decay} to $G$ and its derivatives. The next lemma gives a formula for the Fourier transform of $\Phi$ on $\R^4$.
\begin{lemma}\label{lemma:fourier-transform-radial}
Let $f, g, G$ and $\Phi$ be as above. Then, for all $y \in \R^4$,
\begin{equation}\label{eq:fourier-transform-radial-Phi}
\mathcal{F}_4(\Phi)(y) = -{\int_{\R}{G(\tau)\tau^{-2}e^{ \pi i (-1/\tau) |y|^2}d\tau}}.
\end{equation}
Recall here that  $\lim_{\tau \rightarrow 0}{G(\tau)\tau^{-2}} = 0$, as noted above.
\end{lemma}
\begin{proof}
For $z \in \H$ and $x \in \R^4$, let $\varphi_z(x) = e^{\pi i z|x|^2 }$. Recall that 
\begin{equation}\label{eq:FT-complex-Gaussian-4d}
\mathcal{F}_4(\varphi_z) = (z/i)^{-2}\varphi_{-1/z} = -z^{-2}\varphi_{-1/z}.
\end{equation}
The idea is to add some positive imaginary part to the integration variable $\tau$ in \eqref{eq:definition-Phi}, use the above formula \eqref{eq:FT-complex-Gaussian-4d} and take a limit. To implement it, we write
\begin{align*}
\mathcal{F}_4(\Phi)(y) &= \int_{\R^4}{ e^{- 2\pi i \langle x, y \rangle}\int_{\R}{G(\tau) e^{\pi i \tau |x|^2}d\tau}dx}\\
						&= \int_{\R^4}{ e^{- 2\pi i \langle x, y \rangle}\lim_{b \to 0^{+}}{\int_{\R}{G(\tau) \varphi_{\tau + ib}(x) d\tau}}dx}.
\end{align*}
We want to interchange the limit with the integral over $\R^4$. To justify this, introduce  
\[
U_{b,y}(x) := e^{- 2\pi i \langle x, y \rangle} \int_{\R}{G(\tau) \varphi_{\tau + ib}(x) d\tau} = e^{- 2\pi i \langle x, y \rangle}e^{- b |x|^2}\Phi(x) 
\]
for $b>0$ and $y \in \R^4$. Then $|U_{b,y}(x)| \leq |\Phi(x)|$ for all $x,y,b$ and $|\Phi| \in L^1(\R^4)$ as noted above. Thus, by dominated convergence,
\[
\mathcal{F}_4(\Phi)(y) = \lim_{b \to 0^{+}}{\int_{\R^4}{e^{- 2 \pi i \langle x, y \rangle} \int_{\R}{G(\tau) \varphi_{ \tau + ib}(x) d\tau }dx}}.
\]
Since for each fixed $b> 0$ , we have $|G(\tau) \varphi_{ \tau + ib}(x)e^{- 2 \pi i \langle x, y \rangle} |\leq |G(\tau)|e^{- \pi b |x|^2}$ and the latter is absolutely integrable over $(x,\tau) \in \R^4 \times \R$, we can apply Fubini's theorem and write
\begin{align*}
\mathcal{F}_4(\Phi)(y) &= \lim_{b \to 0^{+}}{\int_{\R}{G(\tau)\int_{\R^4}{e^{- 2 \pi i \langle x, y \rangle} \varphi_{ \tau + ib}(x)dx} d\tau }}\\
& = \lim_{b \to 0^{+}}{\int_{\R}{G(\tau)(-1) (\tau+ib)^{-2}e^{\pi i (-1/(\tau + ib))|y|^2} d\tau }},
\end{align*}
where the last equation follows from \eqref{eq:FT-complex-Gaussian-4d}. Finally, we may again take the limit into the integral by dominated convergence and thus conclude the proof.
\end{proof}

\begin{proof}[Proof of Theorem \ref{thm:application-hup}]
Retain the above set up and notations. We claim that for all $n,m$ such that $n + \varepsilon_n  > 0$, $m + \hat{\varepsilon}_m > 0$, the following equivalences hold:
\begin{align}
\Phi(\sqrt{n+\varepsilon_n}) = 0   \quad \Longleftrightarrow   \quad \widehat{\mu_f}(n + \varepsilon_n, 0) = \int_{\R}{g(\tau) e^{\pi i \tau(n + \varepsilon_n)}d\tau} = 0, \label{eq:orthogonality-1}\\
\mathcal{F}_4(\Phi)(\sqrt{m+\hat{\varepsilon}_m}) = 0   \quad \Longleftrightarrow   \quad \widehat{\mu_f}(0, m + \varepsilon_m) = \int_{\R}{g(\tau) e^{\pi i (1/\tau)(m +\hat{\varepsilon}_m)}d\tau} = 0. \label{eq:orthogonality-2}
\end{align}
To prove this claim, we apply Lemma \ref{lemma:fourier-transform-radial} and integration by parts (as in \cite{HMRV}) to the integrals expressing $\Phi(x)$ and $\mathcal{F}_4(\Phi)(y)$ for $x,y \in \R^{4} \setminus \{0 \}$. Since $G$ vanishes at infinity, we have
\[
\Phi(x) = \frac{-1}{\pi i|x|^2 }\int_{\R}{G'(\tau)e^{\pi i \tau |x|^2}d\tau}= \frac{i}{\pi |x|^2 }\int_{\R}{g(\tau)e^{\pi i \tau |x|^2}d\tau}.
\]
As for $\mathcal{F}_4(\Phi)(y)$, we integrate 
\[
\tau^{-2}e^{\pi i (-1/\tau)|y|^2} = \frac{1}{\pi i |y|^2} \partial_{\tau}e^{\pi i (-1/\tau)|y|^2}
\]
and differentiate $G$ to obtain
\[
\mathcal{F}_4(\Phi)(y) =  \frac{1}{\pi i |y|^2} \int_{\R}{g(\tau)e^{\pi i(-1/\tau)|y|^2}d\tau}.
\]
Thus we deduce, in an explicit manner, the formula in \cite[equation~(1.2.7)]{HMRV} and the claimed equivalences \eqref{eq:orthogonality-1}, \eqref{eq:orthogonality-2} (where we also use that $g$ is even to deduce \eqref{eq:orthogonality-2}). To finish the proof of Theorem \ref{thm:application-hup}, we choose $s_0$ so large that, under the transformations 
\[
f \mapsto g \mapsto G \mapsto \Phi,  \qquad V_{\text{odd}}^{s_0}(\R) \rightarrow V_{\text{even}}^{s_1}(\R) \rightarrow V_{\text{odd}}^{s_2}(\R) \rightarrow V_{\text{rad}}^{s_3}(\R^4),
\]
described above, the parameter $s_3 = s_3(s_0)$ is so large that we can apply Theorem  \ref{thm:radial-perturbation-in-intro} with input $(d,s, \eta) = (4,1,1/2)$, to the function $\Phi \in V_{\text{rad}}^{s_3}(\R^4)$. (This means that we have to take $s_0$ so large that $s_3 \geq 120 =  (2 \cdot 4)(1 + 2 \cdot 4 + 5 + 2 \cdot (1/2))$.) If $s_0$ has this property, then Theorem \ref{thm:radial-perturbation-in-intro} guarantees the existence of $\delta > 0$ such that, if $|\varepsilon_n|+ |\hat{\varepsilon}_n| \leq \delta n^{-7}$ for all $n \geq 1$, our vanishing assumptions on $\widehat{\mu_f}$ and the equivalences in \eqref{eq:orthogonality-1}, \eqref{eq:orthogonality-2} force $\Phi$ to be zero. Therefore, by injectivity of the Fourier transform on $\R$, we have $G = 0$, which implies that $g = G' = 0$, which implies $f = 0$, as desired.
\end{proof}

%

\end{document}